\documentclass[11pt,reqno]{amsart}

\usepackage{amssymb,amsmath,amsthm,newlfont,enumerate}
\usepackage{a4wide}

\newtheorem{thm}{Theorem}[section]
\newtheorem{prop}[thm]{Proposition}
\newtheorem{cor}[thm]{Corollary}

\newtheorem{lem}[thm]{Lemma}

\newcommand{\cD}{{\mathcal{D}}}
\newcommand{\cS}{{\mathcal{S}}}
\newcommand{\cR}{{\mathcal{R}}}
\newcommand{\cH}{{\mathcal{H}}}

\newcommand{\cL}{{\mathcal{L}}}

\newcommand{\cA}{{\mathcal{A}}}

\newcommand{\cC}{{\mathcal{C}}}
\newcommand{\cF}{{\mathcal{F}}}
\newcommand{\cW}{{\mathcal{W}}}

\newcommand\CC{\mathbb{C}}

\newcommand\TT{\mathbb{T}}
\newcommand\DD{\mathbb{D}}

\newcommand{\geqsim}{\,\raisebox{-0.6ex}{$\buildrel > \over \sim$}\,}
\begin{document}

\title[Trace estimates of Toeplitz operators and applications]{Trace estimates of Toeplitz operators on Bergman spaces and applications to composition operators}

 \author{O. EL-Fallah}
 \address{Mohammed V University in Rabat, Faculty of sciences, CeReMAR -LAMA- B.P. 1014 Rabat, Morocco}
 \email{omar.elfallah@gmail.com; o.elfallah@um5r.ac.ma}
 
 \thanks{Research partially supported by "Hassan II Academy of Sciences and Technology" for the first author.}
 
  \author{M. El Ibbaoui}
 \address{Mohammed V University in Rabat, Faculty of sciences, CeReMAR -LAMA- B.P. 1014 Rabat, Morocco}
 \email{elibbaoui@gmail.com}
 










\begin{abstract}
Let $\Omega$ be a subdomain of  $\CC$ and let $\mu$ be a positive Borel measure on $\Omega$. In this paper, we study the asymptotic behavior of  the eigenvalues of  compact Toeplitz operator $T_\mu$ acting on Bergman spaces on $\Omega$. Let $(\lambda _n(T_\mu))$ be the decreasing sequence of the eigenvalues of $T_\mu$ and let $\rho$ be an increasing function such that $\rho (n)/n^A$ is decreasing for some $A>0$. We give an explicit necessary and sufficient geometric condition on $\mu$ in order to have $\lambda _n(T_\mu)\asymp 1/\rho (n)$. 
As applications, we consider composition operators $C_\varphi$, acting on some standard analytic spaces on the unit disc $\DD$. First, we give a general criterion ensuring that the singular values of $C_\varphi$ satisfy $s_n(C_\varphi ) \asymp 1/\rho(n)$. Next, we focus our attention on composition operators with univalent symbols, where we express our general criterion in terms of the harmonic measure of $\varphi (\DD)$. We finally study the case where $\partial \varphi (\DD)$ meets the unit circle in one point and give several concrete examples. Our method is based on upper and lower estimates of the trace of $h(T_\mu)$, where $h$ is suitable concave or convex functions.
\end{abstract}

\maketitle


\section{Introduction}
Spectral properties of Toeplitz operators associated with positive measures  play an important role in spectral theory of several operators: Hankel operators, composition operators and integration operators.  In this paper, we are interested in the  behavior of the eigenvalues of compact Toeplitz  operators acting on analytic spaces on a subdomain $\Omega$ of $ \CC$ with applications to composition operators.\\
Let $\Omega$ be a domain of $\CC$. We denote by  $H(\Omega)$ the class of all holomorphic functions on   $\Omega $. Let $\omega:\ \Omega \to (0,\infty)$ be a continuous weight on $\Omega$. The weighted Bergman space associated with $\omega$ is given by 
$$
\cA^2_{\omega}= \{ f \in H (\Omega):\ \ \| f \| _\omega = \left( \displaystyle \int _\Omega |f(z)|^2dA_\omega(z) \right)^{1/2} <\infty \},
$$
where $dA_\omega (z) = \omega ^2(z)dA(z)$ and $dA$ is the Lebesgue measure on $\CC$.\\
Clearly, $\cA ^2 _\omega$ is a reproducing kernel space. The reproducing kernel of $\cA ^2_\omega$ will be denoted by $K$ ( or $K^\omega$ if necessary).\\
In this paper, we call the standard Bergman spaces, denoted by $\cA^2_\alpha$, the Bergman spaces on $\DD$ associated with $\omega ^2(z):= \frac{\alpha+1}{\pi}(1-|z|^2)^\alpha$, where $\alpha >-1$. The standard Fock spaces $\cF^2_\alpha$ corresponds to $\Omega = \CC$ and $\omega ^2(z) =\frac{\alpha}{\pi} e^{-\alpha |z|^2}$, where $\alpha >0$.\\
The Toeplitz operator $T_\mu$, acting on $\cA ^2_\omega$, induced by a positive Borel measure $\mu$ on $\Omega$ is given by
$$
T_\mu (f) (z) = \displaystyle \int _\Omega f(\zeta)K (z, \zeta)\omega ^2(\zeta)d\mu(\zeta).
$$

The boundedness, compactness and membership to Schatten classes of Toeplitz operators were studied in several papers (see for instance \cite {Has, OP, Lue, LR, IZ, AP, SY, EMMN}). It is proved, under some regularity conditions on $\omega$, that $T_\mu$ is bounded (resp. compact) if and only if $\mu (R_n)/A(R_n)=O(1)$ (resp. $o(1)$), where $(R_n)$ is a suitable lattice of $\Omega$ with respect to $\omega$.\\
\indent Our goal in this paper is to study the asymptotic behavior of the eigenvalues of compact Toeplitz operators on $\cA^2_\omega$. First, we fix some notations. The class of weights on $\Omega$ considered in this paper, denoted by $\cW (\Omega)$, contains all standard weights. Some examples are listed in section \ref{Pre}.
For  $\omega \in \cW (\Omega)$,  we associate  $\cL _\omega $, which consists of suitable lattices  of $\Omega$, with respect to $\omega$, the  definitions of $\cW$  and $\cL _\omega$ are given in section \ref {Pre}.\\
\indent Throughout this paper we suppose that $T_\mu$ is compact. The decreasing sequence of the eigenvalues of $T_\mu$ will be denoted by $(\lambda _n(T_\mu))$. It is proved in \cite {EMMN}, that $ \lambda _n(T_\mu)= O(1/\log ^\gamma (n))$ for some $\gamma >0$ if and only if there exists $c>0$ such that
$$
\displaystyle \sum _n \exp \left ( -c \left (\frac   {A(R_n)}{\mu (R_n)} \right ) ^{1/\gamma} \right )<\infty,
$$
for some $(R_n)_n \in \cL _\omega$.\\
In this paper we are interested in compact Toeplitz operators, $T_\mu$, such that $1/\lambda _n(T_\mu) = O(n^A)$ for some $A>0$.\\

 Recall that since $T_\mu$ is compact, we have $ \displaystyle \lim _{n\to +\infty}\mu (R_n)/A(R_n) =0.$
Let $(R_n(\mu))$ be an enumeration of $(R_n)$ such that the sequence $
a_n(\mu): = \mu (R_n(\mu))/A(R_n(\mu)),
$
is decreasing. First, we will prove the following result.\\

\indent {\bf Theorem A.}
{\it
 Let $(R_n)\in \cL _\omega$, where  $\omega \in \cW$. Let  $\rho: [1,+\infty ) \to (0,+\infty )$ be an increasing function such that $\rho (x)/x^A$ is decreasing for some $A>0$.
Let $\mu $ be a positive Borel measure on $\Omega $ such that $T_\mu$ defines a compact operator on $\cA ^2_\omega$ . Then
\begin{enumerate}
\item [(1)] $\lambda _n(T_\mu) =O\left ( 1/\rho (n) \right )  \iff  a_n(\mu) = O \left (1/\rho (n) \right )$.\\
\item [(2)] $\lambda _n(T_\mu) \asymp 1/\rho (n)  \iff  a_n(\mu) \asymp 1/\rho (n)$.\\
\end{enumerate} 
}
A preliminary version of this theorem, in the case of standard Bergman spaces of the unit disc, was announced in \cite {EE}. Before going on, two remarks on Theorem A are in order.\\
\begin{itemize}
\item [(i)] The growth condition on $\rho$ is, in some sense, necessary. Indeed, let 
 $\rho$ be an increasing function such that $\rho (x) = o (\rho (2x))$ when $x\to +\infty$. One can construct  (see Subsection \ref {Remarks 2}) a Toeplitz operator $T_\mu$ such that for any lattice $(R_n)_n$ we have 
$$
 \displaystyle \limsup _{n\to \infty}\frac{\lambda _{n}(  T_\mu ) }  { a_{n}(\mu)} = +\infty.
$$
where $a_{n}(\mu)$ is the decreasing rearrangement of $(\mu (R_n) /A(R_n))$.\\
\item [(ii)] In general, the sequence $(a_n(\mu))_n$ is not sufficient to give asymptotic estimates of $( \lambda _n(T_\mu))_n$. Indeed, one can construct two positive Borel measures $\mu$ and $\nu$ on the unit disc $\DD$ such that 
$$a_n(\mu) = a_n(\nu)\quad \mbox{and}\quad  \displaystyle \limsup _{n\to \infty}\lambda _n(T_\mu) /\lambda _n(T_\nu) = \infty.$$
\end{itemize}
Next, we analyze the connection between the behavior of the eigenvalues of $T_\mu$ and the behavior of the Berezin transform of $T_\mu$. Recall that the Berezin transform of a Toeplitz operator $T_\mu$ acting on $\cA^2_\omega$ is given by 
$$
{\tilde \mu}(z)= \frac{  \langle T_\mu K_z,  K_z  \rangle  }{\| K_z\| ^2}, \quad (z\in \Omega).
$$
Let $(R_n)_{n\geq 1} \in \cL _\omega$ and let $z_n$ be the center of $R_n$.
It is known that $T_\mu $ is compact if and only if 
$$\displaystyle \lim _{n\to \infty}{\tilde \mu }(z_n)=0.$$
As before, let $(z_n(\mu))$ be an enumeration of $(z_n)$ such that  the sequence $(b_n(\mu))_n$, defined by 
$$b_n(\mu):= {\tilde \mu}(z_n(\mu)),$$
is decreasing.\\

First, we consider Toeplitz operators $T_\mu $ such that $1/\lambda _n(T_\mu) = O(n^\gamma)$ for some $\gamma \in (0,1)$. We have the following\\

{\bf Theorem B.}{\it \ 
Let $\omega \in \cW$ and let $\mu$ be a positive Borel measure on $\Omega$ such that $T_\mu$ is compact. Let $\rho : [1,+\infty) \to (0,+\infty)$ be an increasing positive function such that $\rho (x)/x^\gamma$ is decreasing for some $\gamma \in (0,1)$. We have
$$
\lambda _n(T_\mu) \asymp 1/\rho (n)\quad \iff \quad b_n(\mu) \asymp 1/\rho (n).
$$
}\\
\indent The case  $\lambda _n(T_\mu)  \lesssim 1/n^A $ for some $A>1$, is rather different. Indeed, to have a description of the behavior of the eigenvalues of such Toeplitz operators in terms of $(b_n(\mu))$ it  is  necessary  that 
\begin{equation}\label {NCB}
C_p(\cA^2_\omega, (R_n) ):= \sup _{n\geq 1}\displaystyle \sum _{j\geq 1}{\tilde \nu} ^p _n(z_j) < \infty,\quad (p\in (0,1)),
\end{equation}
where $d\nu _n = dA_{|R_n}$ (see Theorem \ref {Berezinconcave}).\\

We will  prove the following converse.\\

{\bf Theorem C.}{\it \ Let $\omega \in \cW,\ (R_n)_n \in \cL _\omega$. Let $\mu$ be a positive Borel measure on $\Omega$ such that $T_\mu$ is compact. Suppose that $C_p(\cA^2_\omega , (R_n)_n)<\infty$ for all $p\in (0,1)$. Let $\rho : [1,+\infty) \to (0,+\infty)$  be an increasing positive function satisfying  $\rho (t)/t^\gamma$ is increasing for some $\gamma >1$ and $\rho (t)/t^\beta$ is decreasing for some large $\beta$. Then we have
$$
\lambda _n(T_\mu) \asymp 1/\rho (n)\quad \iff \quad b_n(\mu) \asymp 1/\rho (n).
$$
}
The proofs of these theorems are based on  upper and lower estimates of the trace of $h(T_\mu)$ for convex and concave functions $h$.\\

As application, we consider composition operators on $\cH_\alpha = \{f \in H(\DD): \  f' \in \cA^2_\alpha \}$, which was the original motivation of this work. Let $\varphi $ be an analytic self map of $\DD$. The composition operator on $\cH _\alpha$ induced by a symbol $\varphi$ is defined by 
$$C_\varphi (f)= f\circ \varphi, \quad (f\in \cH _\alpha).$$

Using Theorem A and a standard connection between composition operators and Toeplitz operators, we give estimates of the singular values $s_n(C_\varphi, \cH_\alpha)$ of general composition operators $C_\varphi$,  when $1/s_n(C_\varphi, \cH_\alpha)$ doesn't increase faster than all polynomials.
These estimates are given in terms of the mean values of generalized counting function associated with $\varphi$. \\
We also express these estimates in terms of the harmonic measure of $\varphi (\DD)$, when $\varphi$ is univalent and $\varphi (\DD)$ is a Jordan domain.\\

 Next, we consider composition operators induced by  univalent symbol $\varphi$ such that $\partial \varphi (\DD) \cap \partial \DD$ is reduced to one point. Namely,
We suppose that $\partial \varphi (\DD)$  has, in a neighborhood of +1, a polar equation $1-r=\gamma(|\theta|)$, where $\gamma:]0,\pi]\rightarrow]0,1]$ is a differential increasing function with $\gamma(0)=0$,  and satisfying the following conditions
\begin{equation}\label{g1}
\frac{\gamma (t)}{t} \ \mbox{is increasing}\ ,\quad  \gamma '(t)= O(\gamma (t)/t)\ \ (t\to 0^+),
\end{equation}
and
\begin{equation}\label{g3}
 \gamma (t) = O\left (t/\log ^\beta(1/t)\right )\ \ \text{for some}\ \beta > 1/2.\\
\end{equation}
\\
\indent Recall that by Tsuji-Warschwski's theorem, (see \cite {Sha}),  $C_\varphi$ is compact if and only if
\begin{equation*}\label{TW}
\displaystyle \int _0\frac{\gamma(s)}{s^2}ds =\infty.
\end{equation*}
\indent It is proved in \cite {EEN} that the composition operator  $C_\varphi $ on ${\mathcal {H}}_{\alpha}$ is in $p$-Schatten class ($p>0$) if and only if 
\begin{equation}\label{EEN}
\displaystyle \int _0 \frac{e^{-\frac{p\alpha}{2}\Gamma(t)}}{\gamma (t)}dt <\infty,
\end{equation}
where
\begin{equation*}\label{compactnessfunction}
  \Gamma(t)=\frac{2}{\pi}\int_{t}^1 \frac{\gamma(s)}{s^{2}}\,ds.
\end{equation*}

We have the following result.\\

{\bf Theorem D.} {\it  Let $\alpha>0$ and let $\Omega, \gamma, \varphi$ as before.  Suppose that $\displaystyle \int _0 \frac{\gamma (t)}{t^2}dt = \infty$. We have 
\begin{itemize}
\item [(1)]If $\displaystyle \lim _{t\to 0^+}\frac{\gamma(t)\log (1/t)}{t} =\infty $, then
$$   s_n(C_\varphi, \cH_\alpha) = O(1/n^A)\quad \mbox{for all } \ A>0.$$
\item [(2)]If $\displaystyle \frac{\gamma(t)\log (1/t)}{t} = O(1)$, then 
$$
s_n(C_\varphi , \cH _\alpha) \asymp \exp \left (-\frac{\alpha}{2}\Gamma (x_n)\right),
$$
where $x_n$ is given by $\displaystyle \int _{x_n}^2\frac{dt}{\gamma (t)}=n$.\\
\end{itemize}
}
As examples, we obtain
\begin{cor} \label {sing} With the same notations as above, we have
\begin{enumerate}
\item If $\gamma (t)= \kappa t/{\log (e/t)}$ with $\kappa >0$, then 
$$s_n(C_\varphi , \cH _\alpha ) \asymp \frac{1}{n^{{\alpha \kappa }/{2\pi }}}.$$
\item If $\gamma (t) = \kappa {t}/{\log (e/t)\log \log (e^2/t)}$ with $\kappa  >0$, then 
$$s_n(C_\varphi , \cH _\alpha ) \asymp \frac{1}{(\log n)^{\alpha \kappa /\pi} }.$$
\end{enumerate}
\end{cor}
The article is organized as follows: In section 2, we recall some classical results on compact operators and introduce the weighted Bergman spaces considered throughout this paper. In Section 3, we show how to obtain estimates of the eigenvalues of a compact operator from trace estimates. Section 4 is devoted to proving the estimates of $ tr (h (T_\mu)) $ where $ h $ satisfies some concave/convex conditions. It is important to note that the proof presented in this paper, in particular in the concave case, is different from luecking's proof \cite{Lue} and does not require off-diagonal kernel estimates. This section contains the proof of Theorem A. In Section 5, we study  the behavior of the eigenvalues of $T_\mu$ in terms of its Berezin transform. In section 6, we consider composition operators $C_\varphi$ with general symbol $\varphi$ and give estimates of the singular values of $C_\varphi$ in terms of the generalized Nevanlinna function associated with $\varphi$. Section 7, is devoted to composition operators with univalent symbols. We express the asymptotic behavior of the singular values of $C_\varphi$ in terms of the harmonic measure of $\varphi (\DD)$ and we give explicit examples. In the last section we consider examples of composition operators acting on the Hardy space and on the classical Dirichlet space.\\

\noindent {\bf Notations:} Throughout this paper, we will use the following notations
\begin{itemize}
\item $x\lesssim y$ if there exists a constant $C>0$ such that $x\leq Cy$.
\item $x\asymp y$ if $x\lesssim y$ and $y \lesssim x$.
\item $C(x_1,..,x_n)$ is a constant which depends on $x_1,..,x_n$.
\end{itemize}


\section{Preliminaries}\label {Pre}

\subsection{Compact operators}
Let $H$ be a complex Hilbert space and let $T$ be a bounded operator on $H$. The class of compact operators on $H$ will be denoted $\cS _\infty$ or ($S_\infty (H)$ if necessary). Let $T \in \cS _\infty$. The sequence $(s_n(T))_{n\geq 1}$ (or $(s_n(T,H))_n$) denotes the non increasing sequence of eigenvalues of $(T^*T)^{1/2}$. If $T$ is positive $(s_n(T))_{n\geq 1}$ is the sequence of eigenvalues of $T$ and we write in this case  $s_n(T)=\lambda _n(T)$.\\

By the spectral decomposition of compact operators, every compact operator  $T$ on $H$ can be written as follows
$$
Tf= \displaystyle \sum _{n}s _n (T) \langle f, f_n\rangle g_n, \quad (f\in H),
$$
where $(f_n)$ and $(g_n)_{n\geq 1}$ are  orthonormal systems of $H$.\\
 So, it is easy to see that
$$
s _n(T) =\inf \{ \| T-R\|,\ \dim R(H) < n\}.
$$
In particular, if $T$ and $S$ are two compact operators such that $T= XS$, where $X$ is a contraction, then 
$$ 
s _n(T)\leq s _n(S), \ (\mbox{for all }\ n \geq 0).
$$
Recall that a compact operator $T$ on $H$ belongs to the $p-$Schatten class $\cS_p$ (for $p>0$) if 
$$
\| T\|_p := \left ( \displaystyle \sum _{n\geq 1} s_n(T)^p\right )^{1/p} <\infty.
$$
 The following  result is known as the monotonicity Weyl's Lemma. 
\begin{lem}\label {operator}
Let $T,S$ be two positive bounded operators on a complex Hilbert space $H$ such that $T\leq S$. If $S$ is compact, then $T$ is compact and $\lambda _n(T) \leq \lambda _n (S)$ for all $n\geq 1$.
\end{lem}
We will also need the following  general result \cite {Rot}.
\begin{lem}\label {Traceinequality}
Let $(T_n)_{n\geq 1}$ be a sequence of positive compact operators on a Hilbert space $H$ and let $T= \displaystyle \sum _{n\geq 1}T_n$ (with norm-operator convergence). Let $h: [0,+\infty) \to [ 0,+\infty )$ be an increasing function  such that $h(0)=0$. Then 
\begin{itemize}
\item [(1)] If $h$ is convex, then $
 \mbox { Tr}  \left (h(T) \right )\geq \displaystyle \sum _n \mbox{Tr}  (h(T_n)).
$
\item[(2)] If $h$ is concave, then
$
\mbox{Tr}   \left (h(T) \right )\leq \displaystyle \sum _n \mbox{Tr}  (h(T_n)).
$
\end{itemize}
\end{lem}
The following classical result will be used in section \ref {Trace}.
\begin{lem}\label{A0}
Let $p\geq 1$ and let $(a_n)_{n\geq 1}, (b_n)_{n\geq 1}$ be two positive decreasing sequences. Suppose that 
$$
\displaystyle \sum _{k=1}^na_k^{1/p}\leq \displaystyle \sum _{k=1}^nb_k^{1/p}, \quad ( \text{for all}\,\, n\geq 1).
$$
Then, for every increasing positive function $h$ such that $h(t^p)$ is convex, we have
\begin{equation}
\displaystyle \sum _{k=1}^n h\left ( a_n\right ) \leq \displaystyle \sum _{k= 1} ^nh(b_n).
\end{equation}
\end{lem}
\begin{proof}
This is a direct consequence of Corollary 3.3 of Chapter IV  in \cite {GGK}.
\end{proof}


\subsection{Weighted Bergman spaces}\label {W} In this subsection we recall briefly the definition of the class of weights $\cW$ introduced in \cite {EMMN} . Let $\Omega $ be a domain (bounded or not) of $\CC $ and let $\partial \Omega$ denotes the boundary of $\Omega$.  Let $\partial _\infty \Omega = \partial \Omega$ if $\Omega$ is bounded and $\partial _\infty \Omega = \partial \Omega \cup \{\infty\}$ if $\Omega $ is not bounded. Let $\omega$ be a positive continuous weight on $ \Omega $. In what follows, we suppose that the reproducing kernel $K$ of $\cA ^2_\omega$ satisfies  the following two conditions
\begin{equation}\label {neva}
\displaystyle \lim _{z\to \partial _\infty \Omega }\| K_z\| = \infty.
 \end{equation}
 And for every $\zeta  \in \Omega$
 \begin{equation}\label {C2}
| K(\zeta, z)| = o(\|K_z\|)\qquad (z\to \partial _\infty \Omega ).
\end{equation}
Let $$
\tau ^2(z) (=  \tau^2 _\omega (z) ):=  \frac{1}{\omega ^2(z) \|K_z\| ^2}, \ \ \ (z\in \Omega).
$$
We say that  $\omega \in \cW$  (or $\cW (\Omega)$) if there exists constants  $a,C >0$ such that for $z,\zeta \in \Omega$ satisfying $| z-\zeta| \leq a\tau (z)$, we have
 \begin{align} \label {C3}
 \|K _z \|  \|K_\zeta \| \leq C |K(\zeta, z)|, \quad 
\frac{1}{C}\tau  (\zeta ) \leq \tau (z)\leq C\tau  (\zeta ),
\end{align}
and
\begin{align}\label {C4}
\tau (z) = O \left ( \min (1, dist(z,  \partial _\infty \Omega))\right ).
\end{align}
Now, we give some examples
\begin{itemize}
\item Standard Bergman spaces on the unit disc $\DD $.  Let $\alpha >-1$.
$$ {\cA}_{\alpha}^{2}:=\left\{ f \in H(\mathbb{D})\text{ : }\|f\|_\alpha ^2=\int_\mathbb{D}|f(z)|^{2}\,dA_{\alpha}(z)<\infty\right\}, .$$
 The reproducing kernel is given by $$
K^\alpha _z(w)=\displaystyle \frac{1}{ (1-\overline {z}w)^{2+\alpha }    }$$
and 
$$
 \   \tau ^{2}_\alpha (z):= \tau ^2(z)= {(1+\alpha)}(1-|z|^2)^2.
$$
\item Weighted Bergman spaces on $\DD$.
Let $ \cD$ be the class of Oleinik-Perel'man weights on $\DD$ (see \cite {OP, AP, EMMN}). It is easy to see from \cite {LR, AP}, that if $\omega \in \cD$, then $\omega \in \cW$, $$\|K^\omega _z\|^{2}\asymp \omega^{-2}(z)\Delta (\log (1/\omega (z))\ \mbox{ and}\  \tau _\omega (z)^2\asymp \frac{1}{ \Delta (\log (1/\omega (z)))}.$$
For more general situation see \cite{HLS}.
\item Standard Fock spaces. Let $\alpha >0$.
 \begin{equation}\label {F}
\cal F ^2_\alpha := \cal F ^2_\alpha (\CC )= \{ f \in H (\CC): \ \| f\| ^2 = \frac{\alpha}{\pi}\displaystyle \int _{\CC }|f(z)|^2 e ^{-\alpha |z|^2 }dA(z)<\infty \}.
 \end{equation}
 Then the reproducing kernel is given by $K(z,w) = e ^{\alpha z{\bar w} }$ and $\tau (z)\asymp 1$.\\

\item  Weighted Fock spaces.  In this case $\cA^2_\omega$ will be denoted by $\cF ^2_\omega$.\\
Let  $\omega^{2} (z)= e^{-\Psi (| z|^2)}$ be a positive weight on $\CC$.  We say that $\omega \in \cR$ if  $\Psi: [0,+\infty)\to (0,+\infty)  \in \cC ^3$ and satisfies the following conditions
\begin{equation} \label {SY1}
\Psi' >0, \quad \Psi '' \geq 0, \quad \Psi '''\geq 0,
\end{equation}
and
\begin{equation}\label {SY2}
\Phi '' (x) = O(x^{-1/2}\left (\Phi '(x)\right )^{1+\eta} )\quad \mbox{for some} \ \eta <1/2,
\end{equation}
where $\Phi (x)= x\Psi '(x)$. This class of spaces was considered by K. Seip and E. H. Youssfi in \cite {SY}. One can see that since polynomials are dense in $\cF ^2_\omega$, then conditions (\ref{neva}) and (\ref{C2}) are satisfied. It is proved in  \cite {SY} that
$$
\tau_\omega(z)^{-2} =: K(z,z) \omega ^2(z) \asymp \Phi '(|z|^2).
$$
Using Lemma 3.2 of \cite {SY}, it is not hard to prove that $\omega \in \cW$.\\

\end{itemize}
It is  proved in \cite {EMMN}, that if $\omega \in \cW$ then there exist $B_\omega>1 ,\delta_\omega \in (0, a/4B_\omega)$ such that for all $\delta \in (0,\delta _\omega)$ there exists $(z_n)\in \Omega$ such that  
\begin{itemize}
\item $\Omega =\displaystyle  \cup _{n\geq 1}D (z_n, \delta \tau _\omega (z_n))= \cup _{n\geq 1}D (z_n, B_\omega \delta \tau _\omega (z_n))$.\\
\item $D (z_n,  \frac{\delta}{B_\omega } \tau _\omega (z_n)) \cap D (z_m,  \frac{\delta}{B_\omega } \tau _\omega (z_m))= \emptyset $, for $n\neq m$.\\
\item $ z \in D (z_n, \delta \tau _\omega (z_n))$ implies that $D (z, \delta \tau _\omega (z)) \subset D (z_n, B_\omega \delta \tau _\omega (z_n))$.\\
\item There exists an integer $N$ such that every $D (z_n,B_\omega \delta \tau _\omega (z_n))$ cuts at most $N$ sets of the
family $\left (D (z_m,B_\omega \delta \tau _\omega (z_m))\right )_m)$. We say that $\left ( D (z_n,B_\omega \delta \tau _\omega (z_n))\right )_n$ is of finite multiplicity.\\
\end{itemize}
We say that $(R_n)_n \in  \cL _\omega $ if  $(R_n)_n= ( D (z_n, \delta \tau _\omega (z_n)))_n$ and satisfies the above conditions. \\

In the following, we consider $\omega \in \cW$ and suppose that $T_\mu $ is compact. That is  $\mu (R_n)/A(R_n) = o(1)$. As mentioned before $(R_n(\mu))$ will denote an enumeration of $(R_n)_n$ such that $$a_n(\mu) := \mu (R_n(\mu))/A(R_n(\mu)),$$ is decreasing.\\
\begin{lem}\label {restriction} Let $\omega \in \cW$ and let $(R_n)\in \cL _\omega$.
Let $\mu$ be a positive Borel measure on $\Omega$ such that $T_\mu $ is compact on $\cA^2_\omega$. Denote by  $b=  \frac{1+B_\omega}{2}$ and by $\mu _n$ the restriction of $\mu$ to $\displaystyle \cup _{k\leq n}R_k(\mu)$.  Let $\nu _n, \nu$ be the following  measures 
$$
 d\nu _n=\displaystyle \sum _{j=1}^n a_j(\mu)dA_{ |bR_j(\mu)}\ \mbox{and}\  d\nu =\displaystyle \sum _{j\geq1} a_j(\mu)dA_{ |bR_j(\mu)}.
$$
Then, there exists a constant $C= C(\omega, (R_n))>0$ such that
\begin{enumerate}
\item [\mbox{(1)}] $T_{\mu _n}\leq C T_{\nu_n}$ and  $T_{\mu }\leq C T_{\nu}$.\\
\item [(2)] $\| T_\mu\| \leq C a_1(\mu)$.\\
\item [(3)] $ \lambda _k (T_{\mu_n}) \leq C \lambda _k (T_{\nu_n})$, ($k\geq 1$).\\
\item [(4)] $
\lambda _k(T_{\mu _n}) \leq \lambda _k(T_\mu ) \leq \lambda _k(T_{\mu _n}) + Ca_{n+1}(\mu)
$, ($k\geq 1$).
\end{enumerate}
\end{lem}
\begin{proof}
By the subharmonicity inequality applied to the function $z\to |f(z)/K(z,\zeta)|^2$, there exists a constant $C_1>0$, which depends on $\omega$ and $(R_n)_n$, such that 
\begin {equation}\label{IM}
 |f(\zeta)|^2\omega ^2(\zeta) \leq \frac{C_1}{A(R_n)}\displaystyle \int _{b R_n} |f(z)|^2\omega ^2(z)dA(z),\qquad (\zeta \in R_n).
\end{equation}
It gives that
\begin{equation}\label {IP}
\displaystyle \int  _{R_n}|f(\zeta)|^2\omega ^2(\zeta) d\mu (\zeta) \leq C_1a_n(\mu)\displaystyle \int _{bR_n} |f(z)|^2\omega ^ 2(z)dA (z).
\end{equation}
This implies 
\[
\begin {array}{lll}
\langle T_{\mu _n} f, f \rangle &= &\displaystyle \int _\Omega |f(\zeta)|^2\omega ^2(\zeta) d\mu _n (\zeta)\\
& \leq & \displaystyle \sum _{j= 1}^n \displaystyle \int _{R_j(\mu)} |f(\zeta)|^2\omega ^2(\zeta) d\mu (\zeta)\\
&\leq & C_1\displaystyle \sum _{j=1}^n a_j(\mu)\displaystyle \int _{bR_j(\mu)} |f(z)|^2\omega ^2(z)dA(z) \\
& = & C_1\displaystyle \int _\Omega|f(z)|^2\omega ^2(z)d\nu _n(z)\\
&=&\langle C_1T_{\nu _n} f, f \rangle.
\end{array}
\]
This means that $T_{\mu _n} \leq  C_1 T_{\nu _n}$, which proves the part $(1)$ of the lemma. \\

Let $N$ be the multiplicity of  $\left ( D(z_n,B_\omega \delta \tau _\omega (z_n))\right )_n$.
From part $(1)$ we have 
$$0 \leq T_\mu \leq C_1T_\nu \leq NC_1a_1(\mu)Id_{\cA^2_\omega}.$$ Then $\| T_\mu\| \leq NC_1 a_1(\mu)$.\\

Clearly, part $(3)$ is a consequence of part $1$ and Lemma \ref {operator}.\\

Since $\mu_n \leq \mu $, we have $T_{\mu _n} \leq T_\mu$. Then by Lemma \ref {operator}, we get $\lambda _k(T_{\mu _n}) \leq \lambda _k(T_\mu )$.  For the second inequality note that 
$
\lambda _k(T_\mu ) \leq \lambda _k(T_{\mu _n}) + \| T_\mu -T_{\mu _n}\|.
$
Using the part $(2)$, applied to $\mu -\mu _n$, we obtain
$$\| T_\mu -T_{\mu _n}\| =\| T_{\mu-\mu _n}\| \leq Ca_1{(\mu -\mu_{n})}\leq Ca_{n+1}(\mu).$$
Combining the two last inequalities, we obtain 
$ \lambda _k(T_\mu ) \leq \lambda _k(T_{\mu _n}) + Ca_{n+1}(\mu).$
\end{proof}
\section{A general argument}
Let $\beta >0$ and let $\delta >0$. The  function $h_{\beta, \delta}$ defined on $[0,\infty)$ is given by $$ h_{\beta, \delta}(t)= (t^\beta -\delta)^+:= \max (t^\beta -\delta,0) .$$
The functions $h_{\beta, \delta}$, will play an important role in our study. First,  note that $h_{\beta, \delta}$  is convex for $\beta \geq 1$ and if $\beta \in (0,1)$, we have $h_{\beta, \delta}(t^p)$ and   $h_{\beta, \delta}^p$ are convex if and only if $p\geq 1/\beta$.\\
The following two lemmas will be used in the sequel to obtain estimates of eigenvalues of positive compact operator $T$ from upper and lower estimates of the trace of $h(T)$ for some suitable functions $h$. 
\begin{lem}\label {A1}
Let $\beta \in (0,1]$ and let $(a_n)_{n\geq 1}$ be a decreasing sequence. Let $\rho : [1,+\infty) \to (0,+\infty)$ be an increasing positive function such that $\rho (x)/x^\gamma$ is decreasing for some $\gamma \in  (0,1/\beta)$.
Suppose that there exists $ B >0$ such that for every  $\delta \in (0,1)$, we have
\begin{equation}\label {AIconvex}
\displaystyle \sum _{n\geq 1} h_{\beta, \delta}\left (\frac{1}{B\rho (n)}\right ) \leq \displaystyle \sum _{n\geq 1} h_{\beta, \delta}(a_n) \leq \displaystyle \sum _{n\geq 1} h_{\beta, \delta} \left (\frac{B}{\rho (n)}\right ).
\end{equation}
Then $a_n \asymp {1}/{\rho (n)}$.
\end{lem}
\begin{proof}
Without loss of generality, we suppose that $\rho$ is strictly increasing and $\beta =1$. Let $\delta >0$ and let $h_\delta (t) = (t-\delta )^+$. By (\ref {AIconvex}) we have
$$
 \displaystyle \sum _{a_n \geq  2 \delta }a_n  \leq  2 \displaystyle \sum _{a_n \geq 2\delta }(a_n -\delta)   \leq 2 \displaystyle \sum _{n }h_\delta(a_n ) \leq 2 \displaystyle \sum _{n }h_\delta(B/\rho (n) )\leq \displaystyle \sum _{\rho (n) \leq  B/\delta}\frac{2B}{\rho (n)}.
$$
and 
$$
 \displaystyle \sum _{\rho (n) \leq  \frac{1}{2B \delta } }\frac{1}{B\rho (n)}  \leq  2 \displaystyle \sum _{n }(\frac{1}{B\rho (n)} -\delta )^+  \leq 2 \displaystyle \sum _{n }(a_n-\delta )^+ \leq  2 \displaystyle \sum _{a_n \geq   \delta }a_n. 
 $$
 These inequalities can be written as follows
\begin{equation}\label {AI1}
\displaystyle \sum _{\rho (n) \leq 1/2B\delta}\frac{1}{\rho (n)} \lesssim \displaystyle \sum _{a_n \geq \delta }a_n  \lesssim \displaystyle \sum _{\rho (n) \leq 2B/\delta}\frac{1}{\rho (n)}.
\end{equation}
We have
$$ \sum _{\rho (n) \leq x}\frac{1}{\rho (n)}\asymp \frac{\rho^{-1}(x)}{x}$$
Indeed, obviously we have 
$$
\frac{\rho^{-1}(x)}{x}\lesssim  \sum _{\rho (n) \leq x}\frac{1}{\rho (n)}.
$$

\noindent Conversely, using the fact that $\rho (x)/x^\gamma$ is decreasing, we have 
$$\sum _{\rho (n) \leq x}\frac{1}{\rho (n)}= \sum _{\rho (n) \leq x}\frac{n^\gamma}{\rho (n)}\frac{1}{n^\gamma}\lesssim \frac{\left(\rho^{-1}(x)\right)^{\gamma}}{x}\sum_{\rho (n) \leq x}\frac{1}{n^{\gamma}}\asymp \frac{\rho^{-1}(x)}{x}.$$

\noindent Then $$ \sum _{a_n \geq \delta }a_n  \lesssim \frac{\delta}{B}\rho^{-1}\left(\frac{2B}{\delta}\right).$$

\noindent Let $N(\delta):=\text{Card}\{n:\ a_n\geq \delta\} $. Since $\delta N(\delta)\leq \sum_{a_{n}\geq \delta}a_{n}$, we obtain
\begin{equation}\label{maj}
 N(\delta)\lesssim \frac{1}{B}\rho^{-1}\left(\frac{2B}{\delta}\right)
\end{equation}
Let $A>1$. Since $x^{1/\gamma}/\rho^{-1}(x)$ is decreasing, we have $\rho ^{-1} (x/A) \leq A^{-\frac{1}{\gamma}}\rho^{-1}(x)$. Then
$$\sum _{a_n \geq A\delta }a_n\lesssim  \frac{A\delta}{B}\rho^{-1}\left(\frac{2B}{A\delta}\right)\lesssim \left(\frac{ 1}{A}\right)^{\frac{1-\gamma}{\gamma}  }\frac{ \delta }{B}\rho^{-1}\left(\frac{2B}{\delta}\right) $$  
Then, for sufficiently large $A$, we have 
$$\sum _{a_n \geq \delta }a_n \asymp \sum _{\delta \leq a_n \leq A\delta   }a_n. $$
Using the left inequality of (\ref{AI1}), we get 
\begin{equation}\label{min}
\delta \rho ^ {-1}(1/2B\delta) \lesssim \displaystyle \sum _{\rho (n) \leq 1/2B\delta}\frac{1}{\rho (n)} \lesssim  \sum _{\delta \leq a_n \leq A\delta   }a_n \lesssim A\delta  N(\delta).
\end{equation}
Combining \eqref{maj} and \eqref{min} we obtain $a_n \asymp 1/\rho (n)$.
\end{proof}
The following lemma will be used in section \ref{Berezin}.
\begin{lem}\label {A2}
Let $(a_n)_{n\geq 1}$ be a positive decreasing sequence. Let $\rho : [1,+\infty) \to (0,+\infty)$ be an increasing positive function. Suppose  that there exist $\beta >1 $ and $\gamma >1$ such that $\rho (t)/t^\gamma$ is increasing and $\rho (t)/t^\beta$ is decreasing.\\
Let $ p \in (0,1/\beta )$ and suppose that there exists $ B >0$ such that for every increasing concave function satisfying $h(t)/t^p$ is increasing, we have
\begin{equation}\label {AIconcave}
\frac{1}{B}\displaystyle \sum _{n\geq 1} h\left (\frac{1}{\rho (n)}\right ) \leq \displaystyle \sum _{n\geq 1} h(a _n) \leq B\displaystyle \sum _{n\geq 1} h \left (\frac{1}{\rho (n)}\right ).
\end{equation}
Then $a_n \asymp {1}/{\rho (n)}$.
\end{lem}

\begin{proof}
Let $\delta >0$ and let $h$ be the concave function given by 

\[
h(t) =\left \{
\begin{array}{ccc}
& t  &  t\in (0,\delta) \\\\
 & \delta ^{1-p}t^p &  t\geq \delta.\\
\end{array}
\right.
\]
Clearly $h(t)/t^p$ is increasing. Then (\ref {AIconcave}) implies
\begin{multline}\label {AI2}
\frac{1}{B}\left ( \displaystyle \sum _{\rho (n) > 1/\delta}\frac{1}{\rho (n)} + \delta ^{1-p}\sum _{\rho (n) \leq 1/\delta} \frac{1}{\rho(n) ^p} \right )  \leq \displaystyle \sum _{a_n > \delta }a_n + \delta ^{1-p } \displaystyle \sum _{a_n \geq \delta }a_n^p  \\ \leq B\left ( \displaystyle \sum _{\rho (n) > 1/\delta}\frac{1}{\rho (n)} + \delta ^{1-p}\sum _{\rho (n) \leq 1/\delta} \frac{1}{\rho(n) ^p} \right )
 \end{multline}
Let $N(\delta) =\text{Card} \{ n: a_n \geq \delta \}$.
It is clear that 
$\delta N(\delta)\leq\delta ^{1-p } \displaystyle \sum _{a_n \geq \delta }a_n^p  $. Then
$$ \delta N(\delta)\lesssim \displaystyle \sum _{\rho (n) > 1/\delta}\frac{1}{\rho (n)} + \delta ^{1-p}\sum _{\rho (n) \leq 1/\delta} \frac{1}{\rho(n) ^p}  .$$
Using the fact that $\rho (n)/n^\gamma$ is increasing and $\rho (n)^p /n^{p \beta}$ is decreasing with $\gamma  >1$ and $ p\beta \in (0,1)$, we get 
$$
\displaystyle \sum _{\rho (n) \geq 1/\delta}\frac{1}{\rho (n)} \asymp \delta \rho ^{-1}(1/\delta) \ \mbox{and }\ \delta ^{1-p}\sum _{\rho (n) \leq 1/\delta} \frac{1}{\rho(n) ^p} \asymp \delta \rho ^{-1}(1/\delta) .
$$
Then $\delta N(\delta) \lesssim \delta \rho ^{-1}(1/\delta)$, which implies that $a_n \lesssim 1/\rho (n)$. 

For the reverse inequality we repeat the argument used in the proof of Lemma \ref {A1}. Indeed, one can verify that 
$$
\displaystyle \sum _{a_n < \delta /K}a_n +\delta ^{1-p} \displaystyle \sum _{a_n \geq K\delta}a^p _n \leq C(K)\delta \rho ^{-1}(1/\delta),\ \mbox{with}\  \displaystyle \lim _{K\to \infty}C(K) = 0.
$$
So, from (\ref {AI2}) we obtain 
\begin{eqnarray*}
\delta ^{1-p}/B \sum _{\rho (n) \leq 1/\delta} \frac{1}{\rho(n) ^p} & \leq  & \displaystyle \sum _{a_n < \delta }a_n + \delta ^{1-p } \displaystyle \sum _{a_n \geq \delta }a_n^p  \\
&\leq & \displaystyle \sum _{a_n < \delta /K}a_n + \delta ^{1-p} \displaystyle \sum _{a_n \geq K\delta}a^p_n +\displaystyle \sum _{ \delta /K\leq a_n < \delta}a_n+ \delta ^{1-p } \displaystyle \sum _{\delta \leq a_n < K\delta }a_n^p \\
 & \leq  & C(K)\delta \rho ^{-1}(1/\delta) +C(K,p) \delta N(\delta/K),
\end{eqnarray*}
Taking into account that $$\delta ^{1-p}\sum _{\rho (n) \leq 1/\delta} \frac{1}{\rho(n) ^p} \asymp \delta \rho ^{-1}(1/\delta),$$
 we get , for large $K$, that 
$$
\rho ^{-1}(1/\delta) \lesssim N(\delta /K),
$$
which implies that ${1}/{\rho (n)}\lesssim a_n$. This completes the proof.\\
\end{proof}

\section{Estimates of the trace of $ h(T_\mu)$}\label {Trace}
\subsection{Convex case}
The following result is implicitly proved in \cite {EMMN}. Here we give a direct and short proof.
\begin{thm}\label {Traceconvex}
Let $\omega \in \cW$ and let $(R_n)\in \cL _\omega$.  Let $\mu$ be a positive Borel measure on $\Omega$ such that $T_\mu$ is compact on $\cA ^2_\omega$. Let $h$ be a convex increasing function such that $h(0)=0$. We have 
$$ \displaystyle \sum _{n}h\left (\frac{1}{B}a_n(\mu)\right )\leq \displaystyle \sum _{n}h\left (\lambda _n (T_\mu )\right ) \leq \displaystyle \sum _{n}h\left (Ba_n(\mu)\right ).$$
Where  $B$ is a positive constant which depends on $\omega$ and $(R_n)$.
\end{thm}
\begin{proof}
We will use here the same notations as in Lemma \ref {restriction}. Since $\displaystyle \lim _{n\to \infty}a_n(\nu) = 0$, $T_\nu$ is a compact operator.
Let $(f_n)_{n\geq1}$ be an orthonormal basis of  eigenfunctions of $T_\nu$. We have
\begin{eqnarray*}
\displaystyle \sum _{n\geq 1} h(\lambda _n(T_\mu))& \leq &\displaystyle \sum _{n\geq 1} h(C\lambda _n(T_\nu )) \\
& =& \displaystyle \sum _{n\geq 1} h(C\langle T_\nu  f_n, f_n \rangle) \\
& = &  \displaystyle \sum _{n\geq 1} h\left (C\displaystyle \int _\Omega  |f_n(z)|^2   \omega ^2 (z)d\nu(z)  \right) \\
& \leq  &  \displaystyle \sum _{n\geq 1} h\left (\displaystyle \sum _k NCa_k(\mu)\displaystyle \int _{R_k(\mu)} |f_n(z)|^2   dA_\omega(z)  \right ), \\
\end{eqnarray*}
where $N$ is the multiplicity of $(bR_n)_n$.\\
Since $h$ is convex, by Jensen's inequality, we get 
\begin{eqnarray*}
 \sum _{n\geq 1} h(\lambda _n(T_\mu) )& \leq   & \sum_{n,k}h(N^2Ca_k(\mu)) \frac{1}{N}\int _{R_k(\mu)} |f_n(z)|^2 dA_\omega(z) \\
 &= & \displaystyle \sum_{k}h(N^2Ca_k(\mu))  \frac{1}{N}  \int _{R_k(\mu)} \sum _n |f_n(z)|^2 dA_\omega(z) \\
&\lesssim & \sum_{k}h(N^2Ca_k(\mu))     \displaystyle \int _{R_k(\mu)} \| K_z \| ^2dA_\omega(z) \\
&\lesssim &   \sum_{k}h(N^2Ca_k(\mu)).
\end{eqnarray*}
Conversely, let ${\bar \mu} _j =  \mu _{|R_j(\mu)}$ and put ${\bar \mu} = \displaystyle \sum _{j\geq 1}{\bar \mu} _j$. We have  $ T_{{\bar \mu}} \leq N T_\mu$. So, by Lemma \ref {restriction} and Lemma \ref {Traceinequality}, we have
\[
\begin{array}{lll}
 \mbox{Tr} (h(T_\mu)) & \geq  & \mbox{Tr} (h(\frac{1}{N}T_{\bar \mu}))
 \geq   \displaystyle \sum _{j\geq 1}\mbox{Tr } \left (  h(\frac{1}{N}T_{{\bar \mu} _j})\right )\\
& \geq &  \displaystyle \sum _{j\geq 1}  h(\frac{1}{N}\lambda _1(T_{{\bar \mu} _j}))
 =   \displaystyle \sum _{j\geq 1}h( \frac{1}{N}\|T_{{\bar \mu} _j}\|)\\
& \geq &  \displaystyle \sum _{j\geq 1}h\left (\frac{1}{N} \langle T_{{\bar \mu} _j} \frac{K_{z_j}}{\| K_{z_j}\|}, \frac{K_{z_j}}{\| K_{z_j}\|} \rangle \right ) 
\end{array}
\]
where $z_j$ is the center of  $R_j(\mu)$.\\
Now, since 
\begin{eqnarray*}\label {T2}
\langle T_{{\bar \mu} _j} \frac{K_{z_j}}{\| K_{z_j}\|},  \frac{K_{z_j}}{\| K_{z_j}\|}\rangle  =  \displaystyle  \int _{R_j(\mu)} \left | \frac{K_{z_j}(\zeta)}{\| K_{z_j}\|} \right |^2\omega^2(z)d\mu(z) \asymp  a_j(\mu),\\
\end{eqnarray*}
we obtain the second inequality. This ends the proof.
\end{proof}

\subsection{Concave case}
Theorem A will be obtained from the following result.
\begin{thm}\label {TConcave} Let $\omega \in \cW$ and let $(R_n)\in \cL _\omega$.  Let $\mu$ be a positive Borel measure on $\Omega$ such that $T_\mu$ is compact on $\cA ^2_\omega$. Let $h$ be a concave increasing function such that $h(0)=0$. We have 
$$\frac{1}{B}\displaystyle \sum _{n\geq 1}h\left (a_n(\mu)\right )\leq \displaystyle \sum _{n\geq 1}h\left (\lambda _n(T_\mu)\right ) \leq  B\displaystyle \sum _{n\geq 1}\displaystyle \sum _{k\geq 0} h\left (a_n(\mu) e^{-\gamma k}\right ).$$
In addition, if $h(t)/t^{p}$ is increasing for some $p \in (0,1)$, then 
$$
 \displaystyle \sum _ {n\geq 1} h\left (\lambda _n(T_\mu)\right ) \leq \frac{B}{p}\displaystyle \sum _{n\geq 1} h\left (a_n(\mu)\right ),
$$
where  $B, \gamma >0$ are constants which depend only on $\omega$ and $(R_n)$.
\end{thm}

We will need the following lemma in the proof of Theorem \ref {TConcave}.
\begin{lem}\label {IPS}
Let $\omega \in \cW$ and let $(R_n)\in \cL _\omega$. Let $\nu _n= dA _{|R_n}$, then $T_{\nu_n} $is compact on $\cA^2_\omega$ and 
$$
\lambda _k (T_{\nu_n}) \leq Be^{-\gamma k},
$$
where $B,\gamma >0$  depend on $\omega$ and $(R_n)$.
\end{lem}
\begin{proof}
Following the same proof of Theorem 3.8 of \cite {LR}, there exist $B>0$ and $\delta \in (0,1)$ such that 
$$
\| T _{\nu _n}\|_p^p \leq B(1-\delta ^p)^{-1}\leq \frac{C(\delta, B)}{p}, \qquad (p\in (0,1/2)).
$$
It implies that   $k\lambda ^p_k(T_{\nu _n})\ \leq  \frac{C}{p}$, where $C=C(\delta, B)$. Then, for $\frac{1}{p} = \frac{k}{eC}$, we obtain the desired result, with $\gamma = \frac{1}{eC}$.
\end{proof}

\begin{proof}[Proof of Theorem \ref {TConcave}]
We use the same notations as in Lemma \ref {restriction}. Let ${\tilde \mu}$ be the positive Borel measure given by
$$d{\tilde \mu} = \displaystyle \sum_k\frac{1}{NKa_k(\mu)}\mu _{|R_k(\mu)},\quad \mbox{where} \ K>0$$
and with the convention $\frac{1}{NKa_k(\mu)}\mu _{|R_k(\mu)}=0$ if $a_k(\mu)=0$.
By Lemma \ref {restriction},  $T_{{\tilde \mu}}$ is bounded and  $$\| T_{{\tilde \mu}}\| \leq  C\displaystyle \sup _{n}\frac{{\tilde \mu}(R_n)}{A(R_n)} .$$
We have
\begin{eqnarray*}
\frac{{\tilde \mu}(R_n)}{A(R_n)}& = & \frac {1}{NK}\displaystyle \sum _{k:\ R_k\cap R_n \neq \emptyset }\frac{\mu(R_n\cap R_k)}{a_k(\mu)A(R_n)}\\
&\leq & \frac {1}{NK}\displaystyle \sum _{k:\ R_k\cap R_n \neq \emptyset }\frac{\mu(R_k)}{a_k(\mu)A(R_n)}\\
& \asymp &  \frac{1}{NK}\displaystyle \sum _{k:\ R_k\cap R_n \neq \emptyset }\frac{\mu(R_k)}{a_k(\mu)A(R_k)}\\
&\lesssim & \frac{1}{K}.
 \end{eqnarray*}
Then for a large $K$ we have $\| T_{{\tilde \mu}}\| \leq 1$.\\
Let $(f_n)_{n\geq1}$ be an orthonormal basis of $\cA^{2}_\omega$ of  eigenfunctions of $T_\mu$. We have
\begin{eqnarray*}
\displaystyle \sum _{n\geq 1} h(\lambda _n(T_\mu) ) & =& \displaystyle \sum _{n\geq 1} h( \langle T_\mu f_n, f_n \rangle) \\
& = &  \displaystyle \sum _{n\geq 1} h\left (\displaystyle \int _{\Omega} |f_n(z)|^2  \omega ^2(z) d\mu (z) \right) \\
& \geq  &  \displaystyle \sum _{n\geq 1} h\left (\frac{1}{N}\displaystyle \sum _k \displaystyle \int _{R_k(\mu)}    |f_n(z)|^2   \omega ^2(z)d\mu (z) \right ) \\
& = &  \displaystyle \sum _{n\geq 1} h\left (\displaystyle \sum _kC Ka_k(\mu) \displaystyle \int _{R_k(\mu)}     |f_n(z)|^2   \omega ^2(z) d{\tilde \mu}(z)\right ) \\
&\geq &  \displaystyle \sum _{n\geq 1}\displaystyle \sum _kh(CKa_k(\mu)) \displaystyle \int _{R_k(\mu)}   |f_n(z)|^2   \omega ^2(z)d{\tilde \mu}(z)\\
&= & \displaystyle \sum _kh(CKa_k(\mu)) \displaystyle \sum _{n\geq 1} \displaystyle \int _{R_k(\mu)}  |f_n(z)|^2   \omega ^2(z)d{\tilde \mu}(z)\\
&= & \displaystyle \sum _kh(CKa_k(\mu)) \displaystyle \int _{R_k(\mu)}  \|K_z\|^2   \omega ^2(z) d{\tilde \mu}(z)\\
&\asymp &  \displaystyle \sum _kh(a_k(\mu)).\\
\end{eqnarray*}
Which gives the first inequality.\\

Let ${\bar \mu } _j  = \mu _{|R_j(\mu )}$. Since $\mu \leq  \displaystyle \sum _{j=1}^\infty {\bar \mu } _j$,  by Lemma \ref {restriction} and Lemma \ref {Traceinequality} we have 
$$
\mbox{Tr} (h(T_{\mu }))   \leq \mbox{Tr} (h( \displaystyle \sum _{j=1}^\infty T_{{\bar \mu } _j })) \leq   \displaystyle \sum _{j=1}^\infty  \mbox{Tr} (h (T_{{\bar \mu }  _j}) ).
$$
By Lemma  \ref {restriction},  $T_{{\bar \mu } _j} \leq C T_{\nu _j }$, where  $d\nu_j =a_j(\mu ) dA_{|bR_j(\mu)}$. Then by Lemma \ref {IP}, we have   
$$
\mbox{Tr}  (h( T_{{\bar \mu }  _j} )) \leq \mbox{Tr}  \left ( h(CT_{\nu _j}) \right ) \lesssim  \displaystyle \sum _{ k =1}^\infty h \left ( Ca_j (\mu )e^{-\gamma k} \right )  \lesssim  \displaystyle \sum _{k =1}^\infty h \left ( a_j (\mu )e^{-\gamma k} \right ) 
$$ 
Then we obtain the second inequality of Theorem \ref {TConcave}.\\

Now we prove the last inequality of Theorem \ref {TConcave} . Since $h(t)/t^p$ is increasing,  we have $$h(a_j(\mu)e^{-\gamma k}) \leq h(a_j(\mu))e^{-p\gamma k}.$$
 It implies that
\begin{eqnarray*}
\mbox{Tr} (h(T_{\mu })) & \lesssim &  \displaystyle \sum _{j=1}^\infty \displaystyle \sum _{k =1}^\infty h \left ( a_j \mu)e^{-\gamma k} \right ) \\
& \lesssim &  \displaystyle \sum _{j=1}^\infty \displaystyle \sum _{k =1}^\infty h(a_j(\mu))e^{-p\gamma k}\\
& \asymp &\frac{1}{p}  \displaystyle \sum _{j=1}^\infty h(a_j(\mu)).
\end{eqnarray*}
And the proof is complete.
\end{proof}

\subsection{Remarks}\label {Remarks 1}
\begin{enumerate}
\item It is proved by P. Lin and R. Rochberg in \cite {LR} that if $\omega \in \cD$ then $T_\mu \in S_p$ if and only if $(a_n(\mu))_n \in \ell ^p$ for all $p\geq 1$. They also proved, for $p\in (0,1)$, that if $(a_n(\mu))_n \in \ell ^p$ then $T_\mu \in S_p$. Since $\cD \subset \cW$ it is clear from Theorem  \ref {TConcave} that the converse is also true. (see \cite {APP} for radial weights). \\
\item For $\omega \in \cR$, the class of weights introduced by K. Seip and E. H. Youssfi in \cite {SY}, Theorem \ref {TConcave} completes the characterization of membership to Schatten classes given in \cite {SY}.\\
\item  The factor $1/p$ in Theorem \ref {TConcave} can not be replaced by $1/p^{1-\varepsilon}$. Indeed, let $\Omega = \DD$ and let $\omega =1$. Suppose that 
$$ \displaystyle \sum _{n\geq 1}\lambda _n^p(T_\mu) \leq \frac{B}{p^{1-\varepsilon}}\displaystyle \sum _{n\geq 1}a_n^p(\mu), \quad (\forall \ p>0),$$
for every positive Borel measure $\mu$ on $\DD$. Then if $\mu $ is of compact support, we have
$$n\lambda _n^p(T_\mu) \leq  \displaystyle \sum _{j\geq 1}\lambda _j^p(T_\mu) \leq \frac{C}{p^{1-\varepsilon}}, \quad  \forall \ p\in (0,1),$$
for some constant $C>0$. This implies that 
$$
\lambda _n(T_\mu)\leq e ^{-Kn^{\frac{1}{1-\varepsilon}}}.
$$
Now, for $d\mu = dA _{|{D(0,\delta)}}$ where $\delta \in (0,1)$, we have  
$$\lambda _n(T_\mu) = 2\displaystyle \int _0^\delta r^{2n+1}dr \asymp \frac{1}{n+1}\delta ^{2n+2},$$
which gives a contradiction.\\
\end{enumerate}
The following corollary is somewhat more general than Theorem A.
\begin{cor} \label {GA}Let $\omega \in \cW$ and let $(R_n)\in \cL _\omega$.  Let $\mu$ be a positive Borel measure on $\Omega$ such that $T_\mu$ is compact on $\cA ^2_\omega$. For $p\in (0,1)$, we have
$$
p ^{1+\varepsilon}C_1  \displaystyle \sum _{j=1}^n  a^p_j(\mu) \leq  \displaystyle \sum _ {j=1}^n \lambda ^p_j(T_\mu)  \leq \frac{C_2}{p} \displaystyle \sum _{j=1}^n  a^p_j(\mu),
$$
where  $C_1$ is a positive constant which depends on $\varepsilon, \omega, (R_n)$ and $C_2$ is a  positive constant which depends on $\omega$ and $(R_n)$.
\end{cor}
\begin{proof}
Applying Theorem \ref {TConcave} to $T_{\mu_n}$ and taking into account the multiplicity of $(R_n)$, there exists $B_1>0$ which depends only on $\omega$ and $(R_n)$ such that 
$$
\displaystyle \sum _{j=1}^n \lambda _j^p(T_{\mu_n}) \leq \frac{B_1}{p}\displaystyle \sum _{j=1}^n a_j^p(\mu).
$$ 
By Lemma \ref {restriction},  
$$
\lambda _j(T_{\mu }) \leq \lambda _j(T_{\mu _n}) +Ca_{n+1}(\mu).
$$
We obtain 
$$
\displaystyle \sum_{j=1}^n \lambda _j^p(T_{\mu}) \leq \frac{B_1}{p}\displaystyle \sum _{j=1}^n  (a_j^p(\mu)) +C ^p a_{n+1}^p(\mu)) \leq \frac{B_1(1+C)}{p}\displaystyle \sum _{j=1}^n a_j^p(\mu).
$$ 

\noindent Conversely, let $q\in  (0,1)$.  By Theorem \ref {TConcave}, applied to $T_{\mu_n}$, we have
$$
\displaystyle \sum _{j=1}^\infty  \lambda _j^{pq}(T_{\mu_n})  \leq   \frac{B_1}{pq}\displaystyle \sum _{j=1}^n a_j^{pq}(\mu).
$$
Then 
$$
\lambda _j^p(T_{\mu_n})\leq \left ( \frac{B_1}{jpq}\displaystyle \sum _{k=1}^n a_k^{pq}(\mu) \right ) ^{1/q}\leq \left ( \frac{B_1}{q}\right )^{1/q}  \frac{n^{\frac{1}{q}-1}}{p ^{1/q}j^{1/q}}\displaystyle \sum _{k=1}^n a_k^{p}(\mu) 
$$
Then, for $A>0$ we have
$$
 \displaystyle \sum _{j\geq An+1}  \lambda _j^p(T_{\mu_n}))  \leq   \frac{C(B_1,q)}  {p ^{1/q}A^{\frac{1}{q}-1}}  \displaystyle \sum _{k=1}^n a_k^{p}(\mu) \\
$$
Once time again, by Theorem \ref {TConcave}, we have 
\[
\begin{array}{lll}
 1/B \displaystyle \sum _{j=1}^n a_j^p(\mu)& \leq &\displaystyle \sum _{j=1}^\infty \lambda _j^p(T_{\mu_n})\\
 & \leq &\displaystyle \sum _{j=1}^{An} \lambda _j^p(T_{\mu_n}) +\displaystyle \sum _{j=An+1}^\infty \lambda _j^p(T_{\mu_n}) \\
  & \leq &A \displaystyle \sum _{j=1}^{n}\lambda _j^p(T_{\mu_n}) + \frac{C(B_1,q)}  {p ^{1/q}A^{\frac{1}{q}-1}}  \displaystyle \sum _{j=1}^n a_j^{p}(\mu) \\
  & \leq &A \displaystyle \sum _{j=1}^{n}\lambda _j^p(T_{\mu}) + \frac{C(B_1,q)}  {p ^{1/q}A^{\frac{1}{q}-1}}  \displaystyle \sum _{j=1}^n a_j^{p}(\mu) \\
\end{array}
\]
For $q= \frac{1}{1+\varepsilon}$ and for $A$ big enough we obtain the result.
\end{proof}

Now we can state the following important consequence of Corollary \ref {GA}.
\begin{thm}\label {partialsum}
Let $\omega \in \cW$ and let $(R_n)\in \cL _\omega$.  Let $\mu$ be a positive Borel measure on $\Omega$ such that $T_\mu$ is compact on $\cA ^2_\omega$.  Let $h$ be an increasing function on $[0, +\infty)$ such that $h(0) = 0$ and  $h(t^p)$ is convex for some $p>0$. We have
$$
\displaystyle \sum _{j=1}^n h(\frac{1}{B}a_j(\mu))) \leq  \displaystyle \sum _ {j=1}^n h(\lambda _j(T_\mu) ) \leq \displaystyle \sum _{j=1}^n h(Ba_j(\mu)) ,\quad (n\geq 1),
$$
where  $B>0$ is a positive constant which depends on $\omega$, $(R_n)$ and $p$.
\end{thm}
\begin{proof}
This is a consequence of Lemma \ref {A0} and Corollary \ref {GA}.
\end{proof}

\begin{proof}[\bf Proof of Theorem A] We prove $(1)$. Suppose that $a_n(\mu) = O(1/\rho (n))$. Let $p\in (0,1)$ such that $pA<1$. We have 
$$
\displaystyle \sum _{k=1}^n\frac{1}{\rho ^p(k)} \asymp n/\rho ^p(n).
$$
By Corollary \ref {GA}, since $a_n(\mu) = O(1/\rho (n))$, we obtain 
$$
n\lambda _n^p(T_\mu) \leq \displaystyle \sum _{k=1}^n\lambda _k^p(T_\mu) \lesssim \displaystyle \sum _{k=1}^na _k^p(\mu) \lesssim n/\rho ^p(n).
$$
This implies that $\lambda _n(T_\mu) = O(1/\rho (n))$. The reverse implication is obtained in the same way.\\
The second assertion comes from Theorem \ref{partialsum} and Lemma \ref{A1}.
\end{proof}

\subsection{Remarks on Theorem A}\label {Remarks 2}
In this section we provide two examples. The first one, shows that the condition $\rho (x) /x^A$ is decreasing for some $A>0$, is necessary and sharp. And in the second example we show that the sequence $(a_n(\mu))$ is not sufficient, in general, to describe the asymptotic behavior of the eigenvalues of $T_\mu$.\\
\begin{enumerate}
\item The conclusion of Theorem A is not valid if $\rho$ increases faster than all polynomials. Namely, suppose that 
\begin{equation} \label {Rq}
\displaystyle \lim _{x\to +\infty }\frac{\rho(2x)}{\rho (x)} =+\infty.
\end{equation}
Let 
\begin{equation}\label {CE}
d\mu (re^{it})= \frac{1}{\rho \left(1/(1-r)\right)}rdrdt.
\end{equation}
The Toeplitz operator $T_\mu$ defined on the unweighted Bergman space $\cA^2(\DD)$ is compact. 
Since $\mu$ is radial, it is easy to see that $f_n= (n+1)^{1/2}z^n$ is an eigenfunction of $T_\mu$ and  for all $M>1$ we have 
\[
\begin{array}{lll}
\lambda _n ( T_\mu )&= &2\pi \displaystyle \int _0^1r^{2n+1} \frac{1}{\rho \left(1/(1-r)\right)}dr\\
& \geq & 2\pi \displaystyle \int _0^{1-M/n}r^{2n+1} \frac{1}{\rho \left(1/(1-r)\right)}dr\\
&\geq & \frac{C(M)}{\rho (n/M)}.\\
\end{array}
\]
Let $p$ be an integer and let $(R_{n,j}(p))$ denotes the $p-$adic decomposition of $\DD$, that is
$$
R_{n,j}(p) = \Big\{z\in \DD \,;\ 1 -\frac{1}{p^{n}} \leq |z| < 1-\frac{1}{p^{n+1}}\ \text{and}\
\frac{2j \pi}{p^{n+1}} \leq \arg z < \frac{2(j+ 1) \pi}{p^{n+1}}\,\Big\},\ 0\leq j< p^{n+1}.
$$
Let $(R_n)_n$ be a lattice of $\cA ^2(\DD)$. It is clear that for $p$ big enough, then for all $n$ there exists $(k,j)$ such that $R_{k,j}(p) \subset R_n$. Note also that we have $A(R_{k,j}(p))\asymp A(R_n)$. Then we obtain 
$$
a_n(\mu) \lesssim a'_n(\mu), \quad n\geq 1,
$$
where $(a_n(\mu))_n$ (resp. $a'_n(\mu)$) is the decreasing rearrangement of $\left (\frac{\mu (R_n)}{A(R_n)}\right )_n$ (resp. $\left (\frac{\mu (R_{n,j}(p))}{A(R_{n,j}(p))}\right )_n$).
We have 
$$
\frac{   \mu (R_{   n,j   } (p) ) } { A( R_{n,j } (p)) }  \lesssim \frac{1}{\rho (p^n)}.
$$
Then we have 
$$
a_n(\mu) \lesssim a'_n(\mu) \lesssim\frac{ 1}{\rho (n/1+p)},\quad n\geq 1.
$$
Using (\ref {Rq}), we obtain
$$
\displaystyle \lim _{n\to \infty}\frac{\lambda (T_\mu)}{a_n(\mu)} =\infty.
$$
This proves our assertion.\\
\item  Now, we construct two positive Borel measures $\mu$ and $\nu$ on $\DD$ such that 
$$a_n(\mu) = a_n(\nu)\quad \mbox{and}\quad  \displaystyle \limsup _{n\to \infty}\lambda _n(T_\mu) /\lambda _n(T_\nu) = \infty.$$
To this end, let $\mu$ be the measure given in (\ref {CE}) and let $(R_n)_n$ be the dyadic ($p=2$) decomposition of $\DD$. Let $\nu$ be the measure given by
$$\nu = \displaystyle \sum _{n\geq 1} c_n\delta _{w_n},$$
where $(w_n)_n$  is an interpolating separated sequence of $\cA ^2= \cA ^2_0$  \cite {Sei}. The sequence $(w_n)_n$ satisfies 
$$
\left \| \displaystyle \sum c_n \frac{K_{w_n}}{\| K_{w_n}\|} \right \| ^2 \asymp \displaystyle \sum |c_n|^2.
$$
Let $R(w_n)$ be the unique disc $R_k$ such that $w_n \in R_k$ and put $c_n = a_n(\mu) A(R(w_n))$.  Let $\nu _n= \displaystyle \sum _{k\geq n} c_k\delta _{w_k}$, we have 
$$
\lambda _n (T_\nu) \leq \| T_{\nu _n} \| \lesssim  a_n(\mu). 
$$
This implies that $\lim \inf _{n\to \infty} \lambda _n(T_\nu)/\lambda _n(T_\mu) =  0$, while $a_n(\mu) = a_n(\nu )$.

\end{enumerate}
\section{The Berezin Transform}\label {Berezin}
\subsection{Preliminaries}
The Berezin transform of a bounded operator $T$ acting on $\cA^2_\omega$ is defined by 
$$
{\tilde T}(z)= \frac{\langle TK _z,K _z\rangle}{\| K _z\| ^2},\quad (z\in \Omega).
$$
If $T$ is  positive and compact then 
$$
\mbox{Tr} (T)= \displaystyle \int _{\Omega}{\tilde T}(z) \frac{dA(z)}{\tau^2(z)}.
$$
In particular $T\in \cS _1$ if and only if  ${\tilde T}\in L^1(\Omega , \frac{dA}{\tau^2})$.\\
The following general result is standard and is well known (at least for $h(t)=t^p$) (\cite {zhuOTFS}).
\begin{prop}\label {Upperberezin} Let $T$ be a positive compact operator on $\cA ^2_\omega$. Let $h$ be an increasing function such that $h(0) = 0$. We have 
\begin{enumerate}
 \item If $h$ is convex then 
$\displaystyle \sum _{n}h\left (\lambda _n(T) \right )  \geq \displaystyle \int _{\Omega }h\left ( {\tilde T}(z) \right )\frac{dA(z)}{\tau ^2(z)}.
$
  \item  If $h$ is concave then
$\displaystyle \sum _{n}h\left (\lambda _n(T) \right )  \leq \displaystyle \int _{\Omega }h\left ( {\tilde T}(z) \right )\frac{dA(z)}{\tau  ^2(z)}.
$
\end{enumerate}
\end{prop} 
 \begin{proof}
  Let $(f_n)_{n\geq1}$ be an orthonormal basis of $\cA ^2_\omega$ containing a maximal orthonormal system of  eigenfunctions of $T$. Set $\lambda _n=\lambda _n(T)$ and  write 
$$
\langle T K_z,K _z\rangle = \displaystyle \sum _n\lambda _n|\langle f_n,K _z\rangle|^2=   \displaystyle \sum _n\lambda _n|f_n(z)|^2.
$$ 
 If $ h $ is convex, then
 \[
 \begin{array}{lll}
 \displaystyle \int _{\Omega} h({\tilde T}(z)) \frac{dA(z)}{\tau _\omega ^2(z)} &= & \displaystyle \int _{\Omega} h\left (\displaystyle \sum _n\lambda _n \frac{ |f_n(z)|^2}{\| K_z\|^2}\right ) \frac{dA(z)}{\tau _\omega ^2(z)}\\
 &\leq & \displaystyle \int _{\Omega}   \displaystyle \sum _nh(\lambda _n) \frac{ |f_n(z)|^2}{\| K _z\|^2}\frac{dA(z)}{\tau _\omega ^2(z)}\\
  &= & \displaystyle \sum _nh(\lambda _n) \displaystyle \int _{\Omega}   |f_n(z)|^2\omega ^2(z)dA(z)\\
 & = & \displaystyle \sum _n  h\left (\lambda _n  \right ).
 \end{array}
 \]
The concave case is obtained in the same way.
\end{proof}
\subsection{Trace estimates and consequences}
Let $(R_n)\in \cL _\omega$. In the sequel $z_n$ will denote  the center of $R_n$. For the Toeplitz operator $T_\mu$ acting on $\cA^2_\omega$, the Berezin transform of $T_\mu$ will be  denoted by  $ {\tilde \mu}$.  In this section we use the following notation 
 $${\hat \mu}(z_n) = \mu (R_n) /A(R_n).$$
Our goal in this section is to estimate  the eigenvalues of $T_\mu$ in terms of ${\tilde \mu }(z_n)$. To this end, by Lemma  \ref{A1} and Lemma \ref{A2}, it suffices to estimate $\mbox{Tr}\  h(T_\mu)$ in terms of $(h({\tilde \mu}(z_n)))_{n\geq 1}$.
For the convex case we have the following result
\begin{thm}\label {Berezinconvex}
Let $\omega \in \cW$. Let $\mu$ be a positive Borel measure on $\Omega$ and $h$ be a convex increasing function such that $h(0)=0$. Then 
$$
\displaystyle \sum _{n\geq 1}h(\frac{1}{B}{\tilde \mu }(z_n)) \leq \mbox{Tr }(h(T_\mu)) \leq  \displaystyle \sum _{n\geq 1}h(B{\tilde \mu }(z_n)) .
$$
Where $B >0$ doesn't depend on either $\mu$ or $h$.
\end{thm}
\begin{proof}
By Theorem \ref {Traceconvex}, we have 
$$
\displaystyle \sum _{n\geq 1}h(\frac{1}{B_1}{\hat \mu }(z_n)) \leq \mbox{Tr }(h(T_\mu)) \leq  \displaystyle \sum _{n\geq 1}h(B_1{\hat \mu }(z_n)) .
$$
Since  $ {\hat \mu }(z_n) \lesssim {\tilde \mu }(z_n)$ ( \cite {EMMN}), we deduce that $\mbox{Tr }(h(T_\mu)) \leq  \displaystyle \sum _{n\geq 1}h(B{\tilde \mu }(z_n)) $.\\
On the other hand, by $(\ref{IM})$ we have 
$$
|K(z_n,\zeta)|^2\omega^2(z_n)\lesssim \frac{1}{A(R_n)} \displaystyle \int_{R_n}|K(z,\zeta)|^2\omega^2(z)dA(z), \quad \zeta \in \DD.
$$
It implies that
$$
\frac{|K(z_n,\zeta)|^2}{\| K_{z_n}\|^2} \lesssim  \displaystyle \int_{R_n}\frac{|K(z,\zeta)|^2}{\| K_z\|^2}\frac{dA(z)}{\tau _\omega ^2(z)}, \quad \zeta \in \DD,
$$
and
$$
{\tilde \mu }(z_n) \lesssim \displaystyle \int_{R_n}{\tilde \mu }(z)\frac{dA(z)}{\tau _\omega ^2 (z)}.
$$
Since $h$ is convex, by Theorem \ref {Upperberezin} we obtain, for some $c>0$,
$$
\displaystyle \sum _n h(c{\tilde \mu }(z_n)) \lesssim \displaystyle \int_\Omega h({\tilde \mu }(z)) \frac{dA(z)}{\tau ^2_\omega (z)}\leq   \mbox{Tr }(h(T_\mu)).
$$
The proof is complete.
\end{proof}
\begin{proof}[Proof of Theorem B]
By Theorem \ref {Berezinconvex} we have 
$$
\displaystyle \sum _{n\geq 1}h(\frac{1}{B}b_n(\mu)) \leq \displaystyle \sum _{n\geq 1} h(\lambda _n(T_\mu)) \leq  \displaystyle \sum _{n\geq 1}h(Bb_n(\mu)) .
$$
So, by Lemma \ref {A1} we have $b_n (\mu) \asymp 1/\rho (n)$ if and only if $\lambda _n(T_\mu) \asymp 1/\rho (n).$
\end{proof}
Now, we turn to the concave case.
Let $p>0$ and let $d\nu _n = dA_{|R_n}$. Recall that  
$$C_p(\cA^2_\omega, (R_n)) = \sup _{n\geq 1}\displaystyle \sum _{j\geq 1}{\tilde \nu} ^p _n(z_j) \in [0,+\infty).$$

For the concave case we have the following result
\begin{thm}\label {Berezinconcave}
Let $\omega \in \cW$, $(R_n)_n \in \cL _\omega$ and $p\in (0,1)$. The following assertions are equivalent 
\begin{enumerate}
\item $C_p(\cA^2_\omega, (R_n)_n) <\infty$.\\
\item  There exists $B >0$ such that for every Borel positive measure $\mu$ on $\Omega $ and every increasing concave function $h$ such that $h(t)/t^p$ is increasing, we have
$$
\frac{1}{B}\displaystyle \sum _{n\geq 1}h({\tilde \mu }(z_n)) \leq \mbox{Tr }(h(T_\mu)) \leq  B\displaystyle \sum _{n\geq 1}h({\tilde \mu }(z_n)) .
$$
\end{enumerate}
\end{thm}
\begin{proof}
The same argument as before proves that we have  $\mbox{Tr }(h(T_\mu )) \leq  C_2\displaystyle \sum _{n\geq 1}h({\tilde \mu }(z_n))$.\\
By Lemma \ref {IP}, there exists $B>0$ such that  $ \mbox{Tr }T^p_{\nu _j} \leq B/p$. So, it is obvious that the condition $C_1\displaystyle \sum _{n\geq 1}h({\tilde \mu }(z_n)) \leq \mbox{Tr }(h(T_\mu))$, applied with $\mu = \nu _n$ and $h(t) =t^p$, gives that $C_p(\cA^2_\omega, (R_n)_n) <\infty$. \\

Conversely, suppose that $C_p(\cA^2_\omega, (R_n)_n)  <\infty$. A standard computation gives
\[
\begin{array}{lll}
{\tilde \mu} (z) & \lesssim & \displaystyle \sum _{j\geq 1}{\hat \mu }(z_j)\left ( \displaystyle \int _{R_j}\frac{|K(z,\zeta)|^2}{\|K_z\|^2}\omega ^2(\zeta)dA(\zeta)\right ).\\
& \lesssim &  \displaystyle \sum _{j\geq 1}{\hat \mu }(z_j){\tilde \nu}_j(z).\\
\end{array}
\]
Since $h$ is concave and $h(t)/t^p$ is increasing, we have
$$
h({\tilde \mu} (z)) \lesssim \displaystyle \sum _{j\geq 1}h({\hat \mu }(z_j)){\tilde \nu}_j(z)^p.
$$
Consequently
\[
\begin{array}{lll}
\displaystyle \sum _{n\geq 1} h({\tilde \mu} (z_n))&  \lesssim  & \displaystyle \sum _{j\geq 1}h({\hat \mu }(z_j))\displaystyle \sum _{n\geq 1} {\tilde \nu}_j(z_n)^p.\\
& \lesssim &  C_p(\cA^2_\omega, (R_n)_n)  \displaystyle \sum _{j\geq 1}h({\hat \mu }(z_j))\\
& \lesssim &C_p(\cA^2_\omega, (R_n)_n)  \mbox{Tr}(h(T_\mu)).\\
\end{array}
\]
The proof is complete.
\end{proof}
Theorem C is a direct consequence of the following result
\begin{thm}\label {MBconcave}
Let $\omega \in \cW$ and $(R_n)_n$. Let $\mu$ be a positive Borel measure on $\Omega$ such that $T_\mu$ is compact. Let $p\in (0,1)$ such that $C_p(\cA^2_\omega, (R_n)_n))<\infty$. Let $\rho : [1,+\infty) \to (0,+\infty[$ be an increasing positive function. Suppose  that there exist $\beta \in (0,1/p) $ and $\gamma >1$ such that $\rho (t)/t^\gamma$ is increasing and $\rho (t)/t^\beta$ is decreasing. Then
$$
\lambda _n(T_\mu) \asymp 1/\rho (n)\quad \iff \quad b_n (\mu)  \asymp 1/\rho (n).
$$
\end{thm}
\begin{proof}
By Theorem \ref {Berezinconcave} we have 
$$
\frac{1}{B} \displaystyle \sum _{n\geq 1}h(b_n(\mu)) \leq \displaystyle \sum _{n\geq 1} h(\lambda _n(T_\mu) \leq  B\displaystyle \sum _{n\geq 1}h(b_n(\mu)) .
$$
So, by lemma \ref {A2} we have $b_n (\mu) \asymp 1/\rho (n)$ if and only if $\lambda _n(T_\mu) \asymp 1/\rho (n).$
\end{proof}
\subsection{Examples}
Now, we give some examples. 
\begin{itemize}
\item {\bf Standard Fock spaces.}  Let $\alpha >0$. Let $\cF ^2_\alpha$ be the standard Fock space given by (\ref {F}). First, recall that the Berezin transform of $T_\mu$ is given by
$$
{\tilde \mu }(z) = \displaystyle \int _{\CC } e^{-\alpha |z-\zeta | ^2}d\mu (\zeta) \quad (z\in \CC).
$$
For more informations on Fock spaces see \cite {ZhuFock}.\\

We have $
C_p(\cF^2_\alpha, (R_n)_n) <\infty.
$ for all $p \in (0,1)$. Indeed,
\[
\begin{array}{lll}
\displaystyle \sum _{j\geq 1}{\tilde \nu} ^p _n(z_j) & \asymp & \displaystyle \sum _{n}   \left (\displaystyle \int _{R_n} e^{-\alpha |z_j-\zeta|^2}dA (\zeta) \right )^p \\
& \asymp & \displaystyle \sum _{n}   \left (\displaystyle \int _{R_n} e^{-\alpha |\zeta|^2}dA (\zeta) \right )^p \\
& =&  O(\frac{1}{p}).\\
\end{array}
\]
\item { \bf Weighted analytic spaces.}
Let $\Omega$ be a subdomain of  $\CC$ and let $\omega \in \cW$. Let $M>0$. We say that $\omega \in  \cal W _M$   if the reproducing kernel of $\cA^2_\omega$ satisfies
\begin{align}\label {offdiag}
|K(z,\zeta) |\leq C(M)\| K_z\| \| K_\zeta\| \left ( \frac{\min (\tau_\omega (z), \tau _\omega (\zeta))}{|z-\zeta|} \right )^M .
\end{align}
We will denote $\cW _\infty = \cap _{M>0}\cW _M$. Examples of such weights can be found in \cite { AP, SY, HLS}.\\

\begin{prop}
Let $M>1$ and let $\omega \in \cW _M$, For every $(R_n) \in \cL _\omega$, we have 
$$C_p(\cA^2_\omega , (R_n)_n)) <\infty,\quad  (\ \mbox{for all}\ p >1/M).$$
In particular, if $\omega \in \cW _\infty $ then $C_p(\cA^2_\omega , (R_n)_n)) <\infty$, for all $p>0$.
\end{prop}

\begin{proof}

Let $p >1/M$ and let $z_n$ be the center of $R_n$. We have 
$$
\displaystyle \sum _{j\geq 1}{\tilde \nu} ^p _n(z_j)   =   \displaystyle \sum _{j\geq 1} \left ( \displaystyle \int _{R_n}\frac{|K_{z_j}(\zeta)|^2 }{\|K _{z_j}\|^2}dA_\omega (\zeta)\right )^p.\\\
$$
Since $(B_\omega R_j)$ is of finite multiplicity,  $\Lambda _n :=\{ j:\  B_\omega R_j \cap  B_\omega R_n \neq \emptyset \}$ is finite. Then 
\[
\begin{array}{lll}
\displaystyle \sum _{j\in \Lambda _n} \left ( \displaystyle \int _{R_n}\frac{|K_{z_j}(\zeta)|^2 }{\|K _{z_j}\|^2}dA_\omega  (\zeta)\right )^p  &  \lesssim &  \displaystyle \sum _{j\in \Lambda _n}  \left ( \displaystyle \int _{R_n}\|K_\zeta\|^2 dA _\omega(\zeta)\right )^p\\
& \lesssim &  \displaystyle \sum _{j\in \Lambda _n}  \left ( \displaystyle \int _{R_n}\frac{1}{\tau _\omega ^{2}(\zeta)}dA (\zeta)\right )^p\\
& = & O(1).
\end{array}
\]
On the other hand, since $Mp >1$, let $a>0$ such that $(M-a)p=1$. We have

\[
\begin{array}{lll}
 \displaystyle \sum _{j\notin \Lambda _n}  \left ( \displaystyle \int _{R_n}\frac{|K_{z_j}(\zeta)|^2 }{\|K _{z_j}\|^2 } dA_\omega  (\zeta)\right )^p& \lesssim &  \displaystyle \sum 
_{j\notin \Lambda _n}  \left ( \displaystyle \int _{R_n}\| K _\zeta\|^2 \left ( \frac{\min (\tau_\omega (z_j), \tau _\omega (\zeta))}{|z_j-\zeta|} \right )^{2M}dA_\omega (\zeta)\right )^p\\
& \asymp & \displaystyle \sum _{j\notin \Lambda _n}  \frac{\tau_\omega (z_j)^{(2M-2a)p} \tau ^{2ap}_\omega (z_n)}{|z_j-z_n|^{2Mp}}\\
&\lesssim &   \displaystyle \int _{\Omega \setminus R_n}  \frac{\tau_\omega (\zeta)^{(2M-2a)p-2} \tau ^{2ap}_\omega (z_n)}{|\zeta-z_n|^{2Mp}}dA(\zeta)\\
&\lesssim &   \displaystyle \int _{\Omega \setminus R_n}  \frac{\tau ^{2ap}_\omega (z_n)}{|\zeta-z_n|^{2ap+2}}dA(\zeta)\\
&= & O(1).
\end{array}
\]
Which implies that $C_p(\cA^2_\omega , (R_n)_n)<\infty $, whenever $p>1/M$.
\end{proof}
\item {\bf Standard Bergman spaces on $\DD $.}
Let  $\alpha >-1$, let $\omega^{2} _\alpha (z) = (1+\alpha) (1-|z|^2)^\alpha$ and let $\cA^2 _\alpha$  be the associated standard Bergman spaces.
Recall that the kernel of $\cA^2_\alpha$ is given by 
$$
K^\alpha _z(w)=\displaystyle \frac{1}{ (1-z\overline {w})^{2+\alpha }    }.$$
We have the following proposition
\begin{prop}\label {SBB}
Let $(R_n)_n \in \cL _{\omega _\alpha}$and let $p\in (0,1)$. We have
$C_p(\cA^2_\alpha, (R_n)_n)) <\infty $ if and only if $p>\frac{1}{2+\alpha }$.
\end{prop}
\begin{proof}
Let $(R_n) \in \cL_\omega$. We have 
$$
{\tilde \nu_{n} }(z) = \displaystyle \int _{R_n} \frac{(1-|z|^2)^{2+\alpha}}{|1-{\bar z}\zeta|^{4+2\alpha}} dA_\alpha(\zeta).
$$
Then 
\[
\begin{array}{lll}
\displaystyle \sum _{j\geq 1} {\tilde \nu }^p_n(z_j) & = & \displaystyle \sum _{j\geq 1} \left (\displaystyle \int _{R_n} \frac{(1-|z_j|^2)^{2+\alpha}}{|1-{\bar z_j}\zeta|^{4+2\alpha}} dA_\alpha(\zeta)\right )^p\\
& \asymp & \displaystyle \sum _{j\geq 1} \left ( \frac{(1-|z_j|^2)^{2+\alpha}(1-|z_n|^2)^{2+\alpha}}{|1-{\bar z_j}z_n|^{4+2\alpha}} \right )^p\\
& \asymp & \displaystyle \int _{\DD} \left ( \frac{(1-|w|^2)^{2+\alpha}(1-|z_n|^2)^{2+\alpha}}{|1-{\bar w}z_n|^{4+2\alpha}} \right )^p\frac{dA(w)}{(1-|w|^2)^2}\\
& \asymp & \displaystyle \int _{\DD} \frac{(1-|z_n|^2)^{(2+\alpha)p}} {|1-{\bar w}z_n|^{(4+2\alpha)p}}\frac{dA(w)} {(1-|w|^2)^{2-(2+\alpha)p}}.
\end{array}
\]<<
Then, the last integral is uniformly finite if and only if $p>\frac{1}{2+\alpha}$ ( [\cite{DS} Lemma 2 page 32].
\end{proof}
Proposition \ref {SBB} implies that the Berezin transform is not sufficient to describe the behavior of the eigenvalues of Toeplitz operators. In what follows, we consider a modified Berezin transform  which is more appropriate to our problem in this case (see for instance \cite {WX} and \cite {Pau}).\\

Let $T$ be a bounded operator on $\cA ^2_\alpha$ and let  $s >-1$. The modified Berezin transform, $B_{\alpha, s}(T)$, of $T$ is given by 
$$
B_{\alpha, s}(T)(z)= \frac{\langle TK^s_z, K^s _z \rangle}{\| K^s _z\| ^2_{\alpha}}
$$
Let $\tau (z) = (1-|z|^2)$. We have the following general result
\begin{prop} \label {modifyberezin}
Let $\alpha$ and let $s$ such that $s  > \frac{\alpha  -1}{2}$ . Let $T$ be a positive compact operator on $\cA^2_\alpha$. We have
 \begin{enumerate}
 \item
 $$
 \mbox{Tr} (T) \asymp  \displaystyle \int _{\DD} B_{\alpha, s} (T)(z)
 \frac{dA(z)}{\tau ^2(z)}.
 $$
 \item  Let  $h$ be a concave function such that $h(0)=0$. Then 
 $$
 \mbox{Tr}(h(T)) \lesssim  \displaystyle \int _{\DD} h \left ( B_{\alpha, s} (T)(z)\right ) \frac{dA(z)}{\tau ^2(z)} .
$$
\item  Let  $h$ be a convex function such that $h(0)=0$. Then 
 $$
  \displaystyle \int _{\DD} h \left ( B_{\alpha, s} (T)(z)\right ) \frac{dA(z)}{\tau ^2(z)}  \lesssim \mbox{Tr} (h(T)).
$$
\end{enumerate}
All the implied constants depend on $\alpha$ and $s$.
\end{prop} 
 \begin{proof}
Let $f= \displaystyle \sum _{n\geq 0}a_n z^n \in \cA^2_\alpha $. Write $K^s _z(\zeta )= \displaystyle \sum _{n\geq 0} c_n(s)\overline {z}^n \zeta ^n$. It is known that   $$c_n(s) \asymp (1+n)^{1+s}.$$
This implies that
 $$\| K^{s}_{re^{it}}\| ^2_\alpha = \| K^{s}_{r}\|^2_\alpha \asymp \frac{ 1  }{  (1-r)^{2+2s -\alpha} }.$$
 Then we have
 \[
 \begin{array}{lll}
 \displaystyle \int _\DD \frac{|\langle f,K^s _z\rangle|^2}{\|K^s _z\| ^2_\alpha}  \frac{dA(z)}{\tau^2(z)}  & = &  \displaystyle \int _0^1 \left ( \displaystyle \int _0^{2\pi}|\langle f,K^s _{re^{it}}\rangle|^2 \frac{dt}{2\pi}  \right )\frac{2rdr}{\|K^s _r\| ^2_\alpha \tau ^2 (r)} \\
 &\asymp &  \displaystyle \int _0^1 \left ( \displaystyle \sum_{n\geq 0}\frac{ |a_n|^2r^{2n} c_n^2(s)}{(1+n)^{2+2\alpha}} \right )(1-r)^{2s -\alpha}rdr \\
  &\asymp &  \displaystyle \sum_{n\geq 0}\frac{ |a_n|^2 }{(1+n)^{2\alpha -2s }}  \displaystyle \int _0^1 r^{2n+1}(1-r)^{2s -\alpha}dr\\
  &\asymp & \| f\| ^2_\alpha.\\
 \end{array}
 \]
 Let $(f_n)_{n\geq1}$ be an orthonormal basis of $\cA^2_\alpha$ containing a maximal orthonormal system of  eigenfunctions of $T$. Write  
$$
\langle T K^s _z,K^s _z \rangle = \displaystyle \sum _n\lambda _n|\langle f_nK^s _z\rangle|^2,\quad (\lambda _n(T)= \lambda _n).
$$ 
 Then,
 $$
 \displaystyle \int _{\DD} B_{\alpha,s} (T)(z)  \frac{dA(z)}{\tau^2(z)}   = \displaystyle \sum _n\lambda _n \displaystyle \int _\DD \frac{|\langle f_n,K^s _z \rangle|^2}{\|K^s _z\| ^2_\alpha}\frac{dA(z)}{\tau^2(z)} \asymp \displaystyle \sum _n\lambda _n =  \mbox{Tr}(T).
 $$
 To prove (2), Suppose that  $ h $ is concave
 \[
 \begin{array}{lll}
 \displaystyle \int _{\DD} h(B_{\alpha,s} (T)(z)) \frac{dA(z)}{\tau^2(z)}  &= & \displaystyle \int _\DD h\left (\displaystyle \sum _n\lambda _n  \frac{|\langle f_n,K^s _z\rangle|^2}{\|K^s _z\| ^2_\alpha}\right ) \frac{dA(z)}{\tau^2(z)} \\
 &\geqsim & \displaystyle \int _\DD\displaystyle \sum _n  h\left (\lambda _n  \right ) \frac{|\langle f_n,K^s _z\rangle|^2}{\|K^s _z\| ^2_\alpha}  \frac{dA(z)}{\tau^2(z)} \\
  &= &\displaystyle \sum _n  h\left (\lambda _n  \right ) \displaystyle \int _\DD \frac{|\langle f_n,K^s _z\rangle|^2}{\|K^s _z\| ^2_\alpha}  \frac{dA(z)}{\tau^2(z)}.\\
 & \asymp & \displaystyle \sum _n  h\left (\lambda _n  \right )
 \end{array}
 \]
 The convex case is obtained  in the same way.
 \end{proof}

\begin{lem} \label {Kgamma} Let  $\alpha >-1$, and let $(R_n)_n \in \cL _{\omega _\alpha}$. Let $h$ be a concave function such that $h(t)/t^{p}$ is increasing for some $p \in (0,1)$. Let $s  > \frac{1+p\alpha  -2p}{2p}$ and let $\mu$ be a positive Borel measure on $\DD $, then
$$
 \mbox{Tr} \ h(T_\mu) \ \asymp  \ \displaystyle \sum _{n\geq 1}h(B_{\alpha, s}(T_\mu)(z_n)),
$$
where the implied constants depend  on $\alpha,\ s, \ p$ and  $(R_n)_n$.
\end{lem}
\begin{proof}
By Proposition \ref {modifyberezin} we have $$
 \mbox{Tr}(h(T_\mu)) \lesssim  \displaystyle \int _{\DD} h \left ( B_{\alpha, s} (T_\mu)(z)\right ) \frac{dA(z)}{\tau ^2(z)} \asymp \displaystyle \sum _{n\geq 1}h(B_{\alpha, s}(T_\mu)(z_n)).
 $$
Conversely, by Theorem \ref {TConcave}  we have 
$$ \mbox{Tr} \ h(T_\mu)  \asymp  \displaystyle \sum _{n}h\left (a_n(\mu)\right ).$$
So,
it suffices to verify that 
$$
 \displaystyle \int _{\DD}h\left ( B_{\alpha,s}(z)\right )\frac{dA(z)}{\tau^2(z)} \lesssim \displaystyle \sum _{n}h\left (a_n(\mu)\right ).
$$
Let $\nu =  \displaystyle \sum _n a_n(\mu)dA_{\alpha |bR_n}$. By Lemma \ref {restriction} $T_\mu \lesssim  T_\nu$. Then
$$
\langle T_\mu K^s_z,K^s _z\rangle \lesssim \langle T_\nu K^s_z,K^s _z\rangle
\lesssim \displaystyle \sum _n a_n(\mu) \displaystyle \int _{R_n}|K^s _z(w)|^2 dA_\alpha(w).
$$
Using  the concavity of $h$, we get
\begin{eqnarray*}
 \displaystyle \int _{\DD}h\left ( B_{\alpha,s}(T_\mu)(z)\right )\frac{dA(z)}{\tau^2(z)}& =&
\displaystyle \int _{\DD}h\left ( \frac{\langle T_\mu K^s_z,K^s _z\rangle }{\| K^s _z\|^2_\alpha}\right )\frac{dA(z)}{\tau^2(z)} \\
&\lesssim  & 
\displaystyle \int _{\DD}h\left (\displaystyle \sum _n a_n(\mu) \displaystyle \int _{R_n}\frac{|K^s _z(w)|^2 }{\| K^s _z\| ^2}   dA_\alpha (w)   \right )\frac{dA(z)}{\tau^2(z)}\\
&\lesssim &\displaystyle \int _{\DD}\displaystyle \sum _n h \left (  a_n(\mu) \displaystyle \int _{R_n}\frac{|K^s _z(w)|^2 }{\| K^s _z\| ^2}  dA_\alpha (w) \right ) \frac{dA(z)}{\tau^2(z)}.\\
\end{eqnarray*}
On the other hand, we have
\begin{eqnarray*}
\displaystyle \int _{R_n} \frac{ |K^s_z(\zeta)|^2}{\| K^s _z\|^2_\alpha} dA_\alpha (\zeta )& \asymp & \displaystyle \int _{R_n} \frac{ (1-|z|^2)^{2+2s -\alpha}}{|1-\overline{z}\zeta|^{4+2s }}dA_\alpha (\zeta )\\
& \asymp & \frac{ (1-|z|^2)^{2+2s -\alpha}}{|1-\overline{z}z_n|^{4+2s }}(1-|z_n|^2)^{2+\alpha}.\\
\end{eqnarray*}
Using the assumption  $h(t)/t^{p}$ is increasing , we get
\begin{eqnarray*}
 h \left ( a_n(\mu) \displaystyle \int _{R_n} \frac{ |K^s_z(\zeta)|^2}{\| K^s _z\|^2_\alpha}  dA_\alpha (w) \right ) & \lesssim  & h (a_n(\mu)) 
 \left ( \displaystyle \int _{R_n}  \frac{ |K^s_z(\zeta)|^2}{\| K^s _z\|^2_\alpha}  dA _\alpha (\zeta) \right )^p \\
&  \lesssim  & h  \left (a_n(\mu) \right )\left ( \frac{ (1-|z|^2)^{2+2s -\alpha }(1-|z_n|^2)^{2+\alpha}}    {|1-\overline{z}z_n|^{4+2s }}  \right )^p\\
\end{eqnarray*}
Combining all these inequalities, and using the fact that $s > \frac{1-2p+\alpha p}{2p}$, we obtain
\begin{eqnarray*}
\displaystyle \int _{\DD}  h\left ( B_{\alpha,s}(T_\mu)(z) \right )\frac{dAz)}{\tau ^2(z)}&\lesssim & 
\displaystyle \sum _n  h  \left (a_n(\mu) \right ) \displaystyle \int _{\DD}\left ( \frac{ (1-|z|^2)^{2+2s -\alpha }(1-|z_n|^2)^{2+\alpha}}    {|1-\overline{z}z_n|^{4+2s }}  \right )^pdA(z)\\
&\lesssim & \displaystyle \sum _n h  \left ( a_n(\mu) \right ).
\end{eqnarray*}
The proof is complete.\\
\end{proof}

\end{itemize}

Let $(b_n^{\alpha , s}(\mu))_n$ be the decreasing enumeration of $( B_{\alpha,s}(T_\mu)(z_n))_{n\geq 1}$. Theorem C is a direct consequence of the following result.
\begin{thm}\label {MBconcaveS}
Let $\omega \in \cW$ and $(R_n)_n \in \cL _\omega$. Let $\mu$ be a positive Borel measure on $\DD$ such that $T_\mu$ is compact on $\cA^2_\alpha$. Let $\rho : [1,+\infty) \to (0,+\infty[$ be an increasing positive function. Suppose  that there exist $\beta >1$ and $\gamma >1$ such that $\rho (t)/t^\gamma$ is increasing and $\rho (t)/t^\beta$ is decreasing. Then, for $s>\beta +\alpha -2$, we have
$$
\lambda _n(T_\mu) \asymp 1/\rho (n)\quad \iff \quad b_n^{\alpha,s} (\mu)  \asymp 1/\rho (n).
$$
\end{thm}
\begin{proof}
It is a consequence of Theorem \ref{TConcave} and Lemma \ref {Kgamma}.
\end{proof}

\section{Composition operators}
We consider composition operators on weighted analytic spaces on  the unit disc $\DD$. 
Let $\omega \in \cW $,  ${\mathcal {H}}_{\omega}$ will denote the space of analytic functions  $ f\in H(\mathbb{D})$ such that $f' \in\cA^2_\omega$.

The space $\cH _\omega$ becomes a Hilbert space if endowed with the norm $\| . \| _{\cH _\omega}$, given by
$$
\| f\| _{\cH _\omega}^2:=  |f(0)|^2+ \int_\mathbb{D} |f'(z)|^{2}\,dA_{\omega}(z).
$$
For $\omega = \omega  _\alpha $, the space $\cH _{\omega _\alpha} $ will be denote by $\cH _\alpha $.\\

By the classical Littlewood--Paley identity, we have $ {\mathcal {H}}_{1}= H^2$ is the Hardy space. Note also that for $\alpha \in [0,1)$, ${\mathcal {H}}_{\alpha} := \cD _\alpha$ are the weighted Dirichlet spaces and for $\alpha >1$, ${\mathcal {H}}_{\alpha} = \cA ^2_{\alpha -2}$ are the  weighted  standard Bergman spaces. For more informations on these spaces see \cite {Gar, HKZ, EKMR}.\\

Let  $\varphi$ be a holomorphic self map of $ \mathbb{D}$.
 The composition operator $C_{\varphi} $ with symbol $\varphi$  acting on ${\mathcal {H}}_{\omega}$ is defined by  $$C_{\varphi} f =f \circ \varphi ,\quad{ f  \in {\mathcal {H}}_{\omega}}. $$

Several papers gave some general criterions for boudeddness, compactness and membership to Schatten classes of composition operators (see for instance, \cite{ShaA, Lue, ZhuJOT, WX, LLQRMAMS, EKSY, KL}).\\


The Nevanlinna counting function, $ N_{\varphi,\omega}$, of $\varphi$ associated with $\cH _{\omega}$ is defined by
\[
 N_{\varphi,\omega}(w)=\left\{
\begin{array}{ccc}
\displaystyle \sum_{z\in \varphi^{-1}(w)} \omega ^2(z )  \ \in (0,\infty ]&if & w\in \varphi(\DD),\\\\
0&if& w\notin\varphi(\DD).\\
\end{array}
\right.
\]
In what follows, $\mu _{\varphi, \omega}$ will denote the measure given by
$$
d\mu _{\varphi , \omega}(w)=\frac{N_{\varphi , \omega}(w)}{\omega ^2(w)}dA(w),\quad (w\in \DD).
$$
The change of variable formula \cite {Ale}, can be written as follows
$$
\displaystyle \int _\DD |(f\circ \varphi)'(z)|^2 dA_\omega (z)=\displaystyle \int _\DD |f'(z)|^2\omega^2(z) d\mu _{\varphi , \omega}(z).
$$
Using this identity, it is clear that the composition operator $C_\varphi$ on $\cH _\omega$ is closely related to the Toeplitz operator $T_{\mu_{\varphi,\omega}}$ on $\cA^2_\omega$. Indeed, if we suppose that $\varphi (0) =0$. Then the subspace $\cH^0_\omega := \{ f \in \cH _\omega :\ f(0)=0\}$ is  reduced by $C_\varphi$. If $T: \cH^0_\omega \to \cH^0_\omega$, denotes the restriction of $C_\varphi$  to $\cH^0_\omega $, then $T^*T$ is unitarily equivalent to $T_{\mu _{\varphi,\omega}}$ on $\cA^2_\omega$. Namely,
$$
T^*T = V^*T_{\mu_{\varphi, \omega}}V,
$$
where $V f= f'$ is the derivation operator which defines a unitary operator from $\cH^0_\omega$ onto $\cA^2_\omega$. As consequence, we have
\begin{prop}\label{CT}
Let $\varphi$ be an  analytic self map of $\DD$ such that $\varphi (0)= 0$. Then $C_\varphi$ is compact on $\cH _\omega$ if and only if $T_{\mu_{\varphi, \omega}}$ is compact on $\cA ^2_\omega$. In this case, we have 
$$
 s^2_{n}(C_\varphi , \cH _\omega )=  \lambda _n(T_{\mu_{\varphi , \omega}}, \cA^2_\omega).
$$
\end{prop}

As a direct consequence of Proposition \ref{CT} and trace estimates for Toeplitz operators,  we obtain the following results.

\begin{thm} \label {Tcomp}  Let $(R_n)\in \cL _\omega$. Let $p\geq 1$ and $h:\ [0,+\infty) \to [0,+\infty)$ be an increasing function such that $h(t^p)$ is convex and $h(0)=0$. Let $\varphi $ be an analytic self map of  $\DD$ satisfying $\varphi (0)=0$.
We have 
$$ 
\displaystyle \sum _{n} h\left ( \frac{1}{B}\left( \frac{\mu _{\varphi , \omega}(R_{n})}  {A(R_{n})}  \right) \right)
\leq \displaystyle \sum _{n}h\left (s^2 _n(C_\varphi , \cH _\omega ) \right ) \leq  \displaystyle \sum _{n}h\left (B\left( \frac{\mu _{\varphi , \omega}(R_{n})}{A(R_{n})}\right )\right ),
$$
where $B >0$ depends on $\omega $ and $p$.
\end{thm}

\begin{cor}\label {sComposition} Let $\omega \in \cW$ and let $(R_n)\in \cL _\omega$. Let  $\rho: [1,+\infty )  \to (0,+\infty )$ be an increasing function such that $\rho (x)/x^A$ is decreasing for some $A>0$.
Let $\varphi$ be an  analytic self map of $\DD$ such that $\varphi (0)= 0$ and $C_\varphi$ is compact on $\cH_\omega$. Then
\begin{enumerate}
\item $s_n(C_\varphi) =O\left ( 1/\rho (n) \right )\quad  \iff \quad a_n(\mu_{\varphi,\omega}) \asymp O \left (1/\rho ^2 (n) \right ).$
\item $s_n(C_\varphi ) \asymp 1/\rho (n) \quad  \iff \quad a_n(\mu_{\varphi,\omega}) \asymp 1/\rho ^2(n).$
\end{enumerate} 
\end{cor}


\section{Composition operators with univalent symbol on $\cH_\alpha$}  \label{US}

The goal of this section is to provide some concrete examples. We will focus our attention  on composition operators $C_\varphi$ acting on $\cH_\alpha$ such that $\varphi$ is univalent. We will give estimates of the singular values of $C_\varphi$ in terms of the pull-back measure induced by $\varphi $.
\subsection{Composition operators with univalent symbol}
Let $\varphi$ be an analytic self map of $\DD$. The pull-back measure associated with $\varphi$  is the positive borelian measure on $\DD$ defined by 
$$
m_\varphi (B) = m(\{ \zeta \in \TT : \ \varphi (\zeta) \in B\ \}),
$$
where $m$ is the  normalized Lebesgue measure of $\TT$.\\

Let $\Omega$ be a simply connected 
subdomain of $\DD$ which contains $0$.   Let $\varphi $ be a conformal map of $\DD$ onto $\Omega$. 
 Let $\sigma$ be an automorphism of $\DD$. Since $C_\sigma$ is an invertible operator on $\cH_\alpha$, we have   $s_n(C_\varphi,\cH _\alpha) \asymp s_n(C_{\varphi \circ \sigma},\cH _\alpha)\ (n\to \infty)$. So, without loss of generality we suppose, in the sequel, that $\varphi (0)=0$.\\

Let $n,j$ be integers such that $n\geq 1$ and $j\in \{0,2,..,2^n-1\}$. The dyadic square $R_{n,j}$ is given by 
$$
R_{n,j} = \Big\{z\in \DD \,;\ 1 - 2^{-n} \leq |z| < 1-\frac{1}{2^{n+1}}\ \text{and}\
\frac{2j \pi}{2^n} \leq \arg z < \frac{2(j+ 1) \pi}{2^n}\,\Big\}.
$$
By following the same proofs, in all the previous results, one can see that we can replace $(R_n)_n \in \cL _{\omega _\alpha}$ by $(R_{n,j})_{n,j}$.
For our purpose, it is more convenient to consider the Carleson boxes $W_{n,j}$ which are given by
$$
W_{n,j} = \Big\{z\in \DD \,;\ 1 - 2^{-n} \leq |z| \ \text{and}\
\frac{2j \pi}{2^n} \leq \arg z < \frac{2(j+ 1) \pi}{2^n}\,\Big\}.
$$ 
The main result of this section is  the following theorem.

\begin{thm} \label {comp} Let $\varphi $ be a univalent analytic self map of $\DD$.
Let $h:\ [0,+\infty) \to [0,+\infty)$ be an increasing function such that $h(0)=0$. Suppose that there exists $p\geq 1$ such that $ h(t^p)$ is  convex. Let $\alpha >0$, we have 
$$ 
\displaystyle \sum _{n,j}h\left (\frac{1}{B}\left(2^{n}m _{\varphi }(W_{n,j})\right )^\alpha \right ) \leq \displaystyle \sum _{n}h\left (s^2 _n(C_\varphi , \cH _\alpha ) \right ) \leq  \displaystyle \sum _{n,j}h\left (B\left(2^{n}m _{\varphi }(W_{n,j})\right )^\alpha \right ),
$$
where $B>0$ depends on $\alpha $ and $p$.
\end{thm}
Let $( m_n(\varphi))_{n\geq 1} $ be the decreasing enumeration of $  (2^{n}m_\varphi (W_{n,j}) )_{n,j}$. As a consequence of Theorem \ref {comp}, Lemma \ref{A1} and Lemma \ref {A2}, we obtain the following result.

\begin{cor} \label {compcoro}
Let  $\alpha >0$. Let $\varphi$ be a univalent analytic self map of $\DD$. Let $\rho : [1,+\infty) \to (0,+\infty)$ be an increasing function such that $\rho (x)/x^A $ is decreasing for some $A>0$. 
Then the following are equivalent.
\begin{enumerate}
\item $s_n(C_\varphi, \cH _\alpha) \asymp 1/\rho (n)$.
\item  $m_n(\varphi) \asymp 1/\rho ^{2/\alpha} (n)$.
\end{enumerate} 
\end{cor}
To prove Theorem \ref {comp}, we need some intermediate results. We begin by the elementary following lemma.
\begin{lem} \label {cb}
Let $p\geq 1$ and let $h:\ [0,+\infty) \to [0,+\infty)$ be an increasing function such that $h(0) = 0$ and $h(t^p)$ is convex. We have
$$
\sum_{n\geq1}\sum_{j=0}^{2^{n}-1}h\left( C\dfrac{\mu\left(R_{n, j}\right)}{A(R_{n,j})}\right) \leq \sum_{n\geq1} \sum_{j=0}^{2^{n}-1}h\left( 2C \dfrac{\mu\left(W_{n, j}\right)} {  A(W_{n,j})   }\right)\leq B\sum_{n\geq1}\sum_{j=0}^{2^{n}-1}h\left( 4C \dfrac{\mu\left(R_{n, j}\right)} {A(R_{n,j})}\right),
$$
where $B>0$ depends only on $p$.
\end{lem}
\begin{proof}
The first inequality comes from the facts that $h$ is increasing, $R_{n,j} \subset W_{n,j}$ and $A(W_{n,j}) =  2A(R_{n,j})$. For the reverse inequality. We follow the argument given in \cite {LLQRJFA}. We  have
$$W_{n, j}=\bigcup_{l\geq n}\,\bigcup_{k\in H_{l, n, j}}R_{l, k},$$
where $$H_{l, n, j}=\left\{k\in \{0,1,.....,2^{l}-1\};~~\dfrac{j}{2^n}\leq\dfrac{k}{2^l}< \dfrac{j+1}{2^n} \right\}.$$
From the above decomposition and the convexity of $h(t^p)$, we get
\begin{eqnarray*}
\displaystyle \sum_{n=1}^{\infty}\displaystyle \sum_{j=0}^{2^{n}-1}h\left( 2C\dfrac{\mu\left(W_{n, j}\right)}{A(W_{n,j})}\right)&=&\displaystyle \sum_{n=1}^{\infty}\displaystyle \sum_{j=0}^{2^{n}-1} h\left(\displaystyle \sum_{l\geq n}\,\displaystyle \sum_{k\in H_{l, n, j}}2^{2n-2l-1}4C \dfrac{\mu\left(R_{l, k}\right)}{A(R_{l,k})} \right)\\
&\lesssim &  \displaystyle\sum_{n=1}^{\infty}\displaystyle \sum_{j=0}^{2^{n}-1} h \left (  \left (\displaystyle \sum_{l\geq n}\,\displaystyle \sum_{k\in H_{l, n, j}}2^{\frac{2n-2l}{p}}\left(4C \dfrac{\mu\left(R_{l, k}\right)}{A(R_{l,k})}\right)^{1/p}\right )^p\right ) \\
&\lesssim &  \displaystyle\sum_{n=1}^{\infty}\displaystyle \sum_{j=0}^{2^{n}-1} \left (  \displaystyle \sum_{l\geq n}\,\displaystyle \sum_{k\in H_{l, n, j}}2^{\frac{2n-2l}{p}}h\left(4C \dfrac{\mu\left(R_{l, k}\right)}{A(R_{l,k})}\right)\right ) \\
&\leq & \displaystyle \sum_{l=1}^{\infty}\sum_{k=0}^{2^{l}-1}\left ( \displaystyle \sum_{l\geq n}\sum_{k\in H_{l, n, j}}2^{\frac{2n-2l+1}{p}}\right )h\left(4C \dfrac{\mu\left(R_{l, k}\right)}{A(R_{l,k})}\right)\\
&\leq &B \displaystyle \sum_{l=1}^{\infty}\sum_{k=0}^{2^{l}-1}h\left(4C \dfrac{\mu\left(R_{l, k}\right)}{A(R_{l,k})}\right).
\end{eqnarray*}
This ends the proof.\\
\end{proof}

In \cite {LLQRNev}, P. Lefevre, D. Li, H. Queff\'elec and L. Rodr\'iguez-Piazza give an explicit relation between the Nevanlinna counting function of an analytic self map $\varphi $ of $\DD$ and it's pull-back measure. Namely, 
\begin{thm}\label {LLQR}
There exist absolute positive constants  $c_1, c_2, C_1$ and $C_2$  such that for every anlytic self map $\varphi$ of $\DD$, $\zeta \in \TT$ and every $\delta \in (0, \frac{1-|\varphi (0)|}{16})$ one has \\
\begin{enumerate}
\item $N_\varphi (w) \leq C_1m_\varphi (W(\zeta , c_1\delta))$, for every $w\in W(\zeta, \delta)$.\\
\item $m_\varphi (W(\zeta , \delta)) \leq \frac{C_2}{\delta ^2}\displaystyle \int _{W(\zeta, c_2\delta)}N_\varphi (w)dA(w)$.
\end{enumerate}
\end{thm}
In particular we have the following inequalities
\begin{equation} \label {LLQP}
\frac{1}{C_2}m_\varphi (W(\zeta , \delta /c_2))\leq \displaystyle \sup _{z\in W(\zeta , \delta)}N_\varphi  (z) \leq C_1 m_\varphi (W(\zeta, c_1\delta)),
\end{equation}
For a simple proof of these results see  \cite {EK}.\\

We also need a consequence of the well known Hardy-Littlewood inequality.
\begin{lem} \label {HL}
Let $\varphi $ be an analytic self map of $\DD$, let $\alpha >0$ and let $\zeta \in \TT$. There exists an absolute constant $c>0$ such that
$$
m_\varphi (W(\zeta, \delta))^\alpha \leq \frac{C(\alpha)}{\delta^2}\displaystyle \int _{W(\zeta, \kappa \delta) \cap \DD}N^\alpha _\varphi (z)dA(z), \quad (0<\delta <c(1-|\varphi(0)|)),
$$
where $\kappa $ is an absolute constant and $C(\alpha )$ depends only on $\alpha$.
\end{lem}
\begin{proof}
Let $R \in (1,2)$ and let $\psi = \varphi /R$. By Hardy-Littlewood inequality \cite {PP}, for every $z\in \DD$ such that $1-|z|<\frac{1-|\psi (0)|}{2}$ and every $\delta \in (0,1-|z|)$ we have  
\begin{equation}\label {HLIneq}
N_\psi (z)^\alpha \leq \frac{C}{\delta ^2}\displaystyle \int _{D(z,\delta)}N^\alpha _\psi (w)dA(w).
\end{equation}
Let $z\in \DD$  and let $\delta >0 $ such that $\max( 1-|z|, \delta)<\frac{1}{4}(1-|\varphi (0)|)$. Then, for $R = 1+ \frac{1-|\varphi (0)|}{2}$, we have  $\delta<1-|z|/R < \frac{1-|\psi (0)|}{2}$. By (\ref {HLIneq}), we get
\begin{eqnarray*}
N^\alpha_\varphi(z) &=  &N^\alpha _\psi (z/R)\\
&\leq & \frac{C}{\delta ^2}\displaystyle \int _{D(z/R,\delta)}N^\alpha_\psi (w)dA(w)\\
&\leq & \frac{4C}{\delta ^2}\displaystyle \int _{D(z,2\delta)}N^\alpha _\varphi (w)dA(w).
\end{eqnarray*}
Now let $\zeta \in \TT$ and let $\delta < c(1-|\varphi (0)|)$, where $c=\frac{c_1}{4(2+c_1)}$ and $c_1$ is the constant appearing in (\ref {LLQP}).\\  
For  $z\in W(\zeta , \delta /c_1 )$, we have $D(z,2\delta) \subset W(\zeta , (2+1/c_1) \delta )$. Then 
\begin{equation*}
N^\alpha_\varphi(z) \lesssim  \frac{1}{\delta ^2}\displaystyle \int _{W(\zeta , \kappa \delta) }N^\alpha _\varphi (w)dA(w),\quad \kappa = 2+1/c_1.
\end{equation*}
And the result comes from   (\ref {LLQP}).
\end{proof}

Let $c>0$ and let $W= W(\zeta, \delta)$ be a Carleson box. We will denote $W^c= W(\zeta, c\delta)$. 
Theorem \ref {comp} is a direct consequence of Theorem \ref {Tcomp}, Lemma \ref {cb} and the following inequalities
\begin{lem}\label {lempull}
Let $\alpha >0$. Let $h:\ [0,+\infty) \to [0,+\infty)$ be an increasing positive function such that $h(t^p)$ is convex for some $p\geq1$. Let $\varphi$ be a univalent analytic self map of $\DD$ and let $C>0$, we have
$$ \displaystyle \sum_{n\geq 1}\sum_{j=0}^{2^{n}-1}h\left(C_1\dfrac{\mu_{\varphi, \alpha}\left(W_{n, j}\right)}{A(W_{n, j})}\right)  \lesssim \displaystyle \sum_{n\geq 1}\sum_{j=0}^{2^{n}-1} h\left(C\left(2^{n}\,m_{\varphi}\left(W_{n,j}\right)\right)^{\alpha}\right)
 \lesssim \displaystyle \sum_{n\geq 1}\sum_{j=0}^{2^{n}-1}h\left(C_2\dfrac{\mu_{\varphi, \alpha}\left(W_{n, j}\right)}{A(W_{n, j})}\right),$$
 where the implied constants don't depend on $h$.
\end{lem}
\begin{proof}
Since $\varphi$ is univalent,  $N_{\varphi , \alpha }= N^\alpha _\varphi$. Then, by equation (\ref {LLQP}) we have
\[
\begin{array}{lll}
\dfrac{\mu_{\varphi, \alpha}\left(R_{n, j}\right)}{A(R_{n, j})} & =  &\dfrac{1}{A(R_{n, j})}\displaystyle \int _{R_{n,j}}\frac{N^\alpha_\varphi(z)}{(1-|z|^2)^\alpha}dA(z)\\
&\lesssim & 2^{\alpha n}\displaystyle \sup _{z\in W_{n,j}}N^\alpha _\varphi(z)\\
& \lesssim& \left (2^{ n}m_\varphi(W^{c_2}_{n,j})\right )^\alpha.
\end{array}
\]
Then 
\[
\begin{array}{lll}
\displaystyle \sum_{n\geq 1}\sum_{j=0}^{2^{n}-1} h\left ( \dfrac{\mu_{\varphi, \alpha}\left(R_{n, j}\right)}{A(R_{n, j})} \right )&  \lesssim &  \displaystyle \sum_{n\geq 1}\sum_{j=0}^{2^{n}-1} h\left( C'_1\left(2^{n}\,m_{\varphi}\left(W^{c_2}_{n,j} \right)\right)^{\alpha} \right)\\
&  \lesssim &  \displaystyle \sum_{n\geq 1}\sum_{j=0}^{2^{n}-1} h\left( C_2\left(2^{n}\,m_{\varphi}\left(W_{n,j} \right)\right)^{\alpha} \right).\\
\end{array}
\]
Then the left inequality of Lemma \ref {lempull} is obtained from Lemma \ref {cb}.\\
Conversely, by Lemma \ref {HL}, we have
 $$
\left ( 2^nm_\varphi (W_{n,j}) \right )^\alpha \lesssim \dfrac{\mu_{\varphi , \alpha}(W^\kappa_{n,j})}{A(W^\kappa_{n,j})}.
 $$
Which gives the remaining inequality in order to finish the proof.
\end{proof}


\subsection{Examples} \label {point}
Let $\Omega$ be a subdomain of $\DD$ such that $0 \in \Omega$,  $\partial \Omega \cap \partial \mathbb{D}= \{1\}$ and $\partial \Omega$ has, in a neighborhood of +1, a polar equation $1-r=\gamma(|\theta|)$, where $\gamma:[0,\pi]\rightarrow[0,1]$ is a differential continuous increasing function such that $\gamma(0)=0$  and satisfying 
 $\gamma '(t)= O(\gamma (t)/t)\ \ (t\to 0^+)$.\\
 Let $\varphi$  be a univalent  map from $\mathbb{D}$ onto $\Omega$ with $\varphi (0)= 0$ and $\varphi(1)=1 $. By definition, the harmonic measure $\varpi(., E, \Omega)$ is the harmonic extension  of $\chi _E$ on $\Omega $, where $E$ is closed subset of $\partial \Omega$. By conformal invariance of the harmonic measure we have
 $$
\varpi(0, E, \Omega) = \varpi(0, \varphi ^{-1}(E), \DD)=m(\varphi ^{-1}(E))= m_\varphi (E).
$$
 So to use Theorem \ref {comp}, we  have to estimate the harmonic measure of our domains. To this end we use  Ahlfors-Warschawski type estimates. The following lemma, is proved in  \cite {EEN}. In the sequel of this subsection, we  suppose that $\gamma$ satisfies conditions (\ref {g1}) and (\ref {g3}).

\begin{lem} \label {esthar}
Let $\gamma, \Omega$ and $\varphi$ as above. Then  
\begin{itemize}
\item $
\varpi(0, W_{n, j}\cap \partial \Omega, \Omega) \lesssim \dfrac{C}{2^n}  \exp \left [ -\Gamma\left(\frac{2\pi (j+1)}{2^n}\right)\right], \quad (0\leq j < j_n:=\frac{2^{n}}{2\pi}\gamma^{-1}(2\pi/2^n)).
$
\item There exists $\eta >0$ such that for $k$ satisfying $2^{k+1}\leq j_n$, we have
$$
\text{Card}\  \left \{j \in \{2^k,.., 2^{k+1}-1\}:\ \dfrac{\eta}{2^n}  \exp \left[-\Gamma\left(\frac{2\pi (j+1)}{2^n} \right) \right] \lesssim  \varpi(0, W_{n, j}\cap \partial \Omega, \Omega) \right \} \asymp 2^k ,
$$
\end{itemize}
\end{lem}
Now, we are able to prove the following estimates

\begin{thm} \label {univalent}  Let $\gamma, \Omega $ and $\varphi$ as above. Let $h:\ [0,+\infty) \to [0,+\infty)$ be an increasing function such that $h(0) =0$. Suppose that there exists $p\geq 1$ such that $h(t^p)$ is convex. Let  $\alpha >0$, we have
$$ 
B\displaystyle \int _0^1 \frac{h(be ^{-\alpha \Gamma (s)})}{\gamma (s)}ds \leq \displaystyle \sum _{n}h\left (s^2 _n(C_\varphi , \cH _\alpha ) \right ) \leq  A\displaystyle \int _0^1 \frac{h(ae ^{-\alpha\Gamma (s)})}{\gamma (s)}ds,
$$
 where $A, B, a, b >0$ depend on $\alpha$  and $p$.
\end{thm}
\begin{proof}
By Theorem \ref {comp} and Lemma \ref {esthar}, it suffices to prove that
$$  \displaystyle \int_{0}^{1}\dfrac{h\left(\frac {C}{A}e^{-\alpha \Gamma(s)}\right)}{\gamma(s)}\,ds \lesssim  \displaystyle \sum_{n = 1}^\infty \displaystyle \sum_{j = 0}^{j_n} h\left(  C  \exp \left[ -\alpha \Gamma\left ( \frac{2\pi (j+1)}{2^n} \right)  \right]  \right) \lesssim \displaystyle \int_{0}^{1}\dfrac{h\left(ACe^{-\alpha \Gamma(s)}\right)}{\gamma(s)}\,ds.$$
First, remark that  
\begin{equation}
\displaystyle \int _{2\pi j/2^n}^{2\pi (j+1)/2^n}\dfrac{\gamma (s)}{s^2}ds\lesssim \displaystyle \int _{2\pi j/2^n}^{2\pi (j+1)/2^n}\dfrac{1}{s}ds= O(1).
\end{equation}  
So, there exists $A>0$ such that 
$$\frac{1}{A}e^{-\alpha \Gamma (2\pi j/2^n)} \leq e^{-\alpha \Gamma (s)} \leq Ae^{-\alpha \Gamma (2\pi j/2^n)}, \qquad s\in \left ( \frac{2\pi j} {2^n}, \frac{2 \pi (j +1)}  {2^n} \right ).$$
Then

$$h\left (Ce^{-\alpha \Gamma(\frac{2\pi j }{2^n})}\right )  \lesssim 2^n \int_{\frac{2j\pi}{2^n}}^{\frac{2(j+1)\pi}{2^n}}h\left ( Ce^{-\alpha \Gamma(s)}\right )ds \lesssim h\left ( ACe^{-\alpha \Gamma(\frac{2\pi (j+1) }{2^n})}\right ). $$
\begin{align*}
 \displaystyle \sum _0^{\infty} \displaystyle \sum_{j=1}^{j_n} h\left(C_1e^{-\alpha \Gamma((\frac{2\pi j}{2^n})}\right)& \geq   \displaystyle \sum _0^{\infty} \displaystyle \sum_{j=1}^{j_n} 2^n  \int_{\frac{2\pi (j-1)}{2^n}}^{\frac{2\pi j}{2^n}}  h\left(\frac{C_1}{A}e^{-\alpha \Gamma(s)}\right)\,ds\\
& \asymp \displaystyle \sum _0^\infty 2^n \int_{0}^{ \gamma ^{-1}(2^{-n})}  h\left(\frac{C_1}{A}e^{- \alpha \Gamma(s)}\right)\,ds\\
& \asymp \displaystyle C \sum _{n=0}^{\infty}2^n\sum _{k=n}^{\infty}  \int_{\gamma^{-1}(2^{-k-1})}^{ \gamma ^{-1}(2^{-k})}  h\left(\frac{C_1}{A}e^{-\alpha \Gamma(s)}\right)\,ds \\
& \asymp \displaystyle \sum _{k=0}^{\infty}2^k \int_{\gamma^{-1}(2^{-k-1})}^{ \gamma ^{-1}(2^{-k})}  h\left(\frac{C_1}{A}e^{-\alpha \Gamma(s)}\right)\,ds\\& \asymp    \int_{0}^{1}\dfrac{h\left(\frac{C_1}{A}e^{-\alpha \Gamma(s)}\right)}{\gamma(s)}\,ds.
\end{align*}
This proves the first inequality.\\
The second inequality can be obtained using similar computations.\\
\end{proof}

\begin{proof}[Proof of \text{Theorem D}]

The first assertion is a direct consequence of the characterization of membership to $p-$Schatten classe given in \cite {EEN}. Indeed, suppose that $\displaystyle \lim _{t\to 0^+}\frac{\gamma(t)\log (1/t)}{t} =\infty $. It is easy to verify that
$$
\displaystyle \int _0 \frac{e^{-\frac{p\alpha}{2}\Gamma(t)}}{\gamma (t)}dt <\infty,\quad (\forall \ p>0).
$$
Then $C_\varphi \in \displaystyle \cap _{p>0}\cS _p(\cH _\alpha)$. This is equivalent to $s_n(C_\varphi,\cH _\alpha) = O(1/n^A) $ for all $A>0$.\\
To prove the second assertion, let 
$$\rho (x) = \exp \left \{ \alpha \Gamma (\Lambda ^{-1}(x) )\right \}, \quad  \mbox{where}\ \Lambda (t) = \displaystyle \int _{t}^2\frac{ds}{\gamma (s)}.$$
First, we prove that $\rho (x)/x^A$ is decreasing, where $A$ is such that $\gamma (t) \leq \frac{\pi A }{2\alpha }\frac {t} {\log (1/t)}$. Since $\gamma (t)/t$ is increasing, we have
$$
\Lambda (t) = \displaystyle \int _t^2\frac{dt}{\gamma(t)} \leq \frac {t} {\gamma (t)}\log (1/t) \leq \frac{\pi A }{2\alpha }\frac {t^2} {\gamma (t)^2}.
$$
This implies that $t\to \Lambda (t) \exp (-\frac {\alpha}{A}\Gamma (t))$ is decreasing and then $\rho (x)/x^A$ is decreasing.\\
 Note also that if $h$ an increasing positive function, then
$$ \displaystyle \int _0^1 \frac{h(Ce ^{-\alpha\Gamma (s)})}{\gamma (t)}dt \asymp \sum_{n\geq 1}h\left (C e^{ -\alpha \Gamma (x_n) }\right )   \displaystyle \int _{x_{n+1}}^{x_n}\frac{dt}{\gamma (t)}= \sum_{n\geq 1}h\left (C e^{ -\alpha \Gamma (x_n) }\right ).$$
Then, by Theorem \ref {univalent} and Lemma \ref{A1} we obtain the result.

\end{proof}


\section{Concluding Remarks}
\subsection{Composition operators on the Hardy space}
The Hardy space $H^2$ is equal to $\cH_1$.  The problem of estimating the singular values of composition operators on $H^2$ was considered in several papers (\cite {LLQRJFA, LLQRArkiv, LLQRA,QS}). Using the same arguments as those given in section \ref{US}, one can remove the condition that $\varphi$ is univalent in Corollary \ref {compcoro}. We have the following result.
\begin{thm}\label {1Hardy}
 Let $\varphi $ be an analytic self map of $\DD$ Let $\rho : [1,+\infty) \to (0,+\infty[$ be an increasing function such that $\rho (x)/x^A$ is decreasing for some $A>0$. Then 
$$s_n(C_\varphi, H^2) \asymp 1/\rho (n) \quad \iff  \quad m_n(\varphi) \asymp 1/\rho ^2 (n).$$
\end{thm}
Note that our method can also be applied to composition operators with outer symbol. Such composition operators was considered in \cite{LLQRJFA, BEKM, QS}. Namely, let $\varphi $ be the outer function given by 
\begin{equation}\label {phi_U}
\varphi (z) = \exp \left ({-\displaystyle \int _\TT \frac{e^{it}-z}{e^{it}+z}U(|t|)}\frac{dt}{2\pi}\right ),
\end{equation}
where $U: [0,\pi] \to [0,\infty)$ is an incresing  integrable function such that $U(0)=0$. It is proved in \cite {BEKM, QS}, under some regularity conditions on $U$, that $C_\varphi $ is compact if and only if 
$$
\displaystyle \int _0^1\frac{U(s)}{s^2}ds=+\infty.
$$
It is also proved in \cite {BEKM} that $C_\varphi \in  S_p(H^2)$ if and only if 
$$
\displaystyle \int _0^1\frac{dt}{U(t)q^{\frac{p}{2}-1}_U(t)}<\infty,
$$
where $q_U(t) = \displaystyle \int _t^1\frac{U(s)}{s^2}ds$.\\
One can extends this result. Namely, in accordance with \cite{BEKM}, we say that $U$ is admissible if $U$ is concave or convex and if $ U(t) \asymp U(2t) \asymp tU'(t)$. We have 
\begin{thm} \label {GBEKM}
Let $U$ be an admissible function  such that $t^2=o(U(t))$ and 
$$
U(t ) = \left (  t \displaystyle \int _t ^{\pi} \frac{U(s)}{s^2}ds \right )\ \  (t\to 0^+).
$$

Let $h$ be an increasing function such that $h(0)=0$. Suppose that there exists $p \geq 1$ such that $h(t^p)$ and $h^p$ are convex. We have
$$
B\displaystyle \int _0^1h\left (\frac{b}{q_U(t)}\right )\frac{q_U(t)}{U(t)}dt\leq \displaystyle \sum _{n}h\left (s^2 _n(C_\varphi , H^2) \right ) \leq  A \displaystyle \int _0^1h\left (\frac{a}{q_U(t)}\right )\frac{q_U(t)}{U(t)}dt,
$$
 where $\varphi $ is given by (\ref {phi_U}) and $A, B, a, b >0$ depend on $p$.
\end{thm}
In \cite {QS}, H. Queffelec and K. Seip give some estimates of the singular values of such composition operators. They proved that if $U$ is sufficiently regular and $q_U(t) = O(\log ^\gamma \log (1/t))$ for some $\gamma >0$, then 
$$
s_n(C_\varphi,H^2) \asymp \frac{1}{ \sqrt{q_U(e^{-\sqrt{n}})}}.
$$
The extremal decreasing case corresponds to  $q_U(t) = \log ^\gamma \log (1/t)$, they obtained that 
$$
s_n(C_\varphi, H^2) \asymp \frac{1}{\log ^{\gamma /2}n}.$$

Using Theorem \ref{GBEKM} and Lemma \ref {A1}, we extend this result as follows
\begin{thm}
Under the same hypothesis of Theorem \ref {GBEKM}. Suppose that $\displaystyle \int _0^1\frac{U(s)}{s^2}ds=+\infty.
$. We have
\begin{enumerate}
\item If $\displaystyle  \lim _{t\to 0^+}\frac{ \log q_U(t)}{\log \log 1/t} = \infty$, then 
$$s_n(C_\varphi, H^2) = O(1/n^A)\quad (\mbox{for all} \ A>0).$$
\item If $\displaystyle \frac{ \log q_U(t)}{\log \log 1/t}  = O(1)$, then 
$$
s_n(C_\varphi, H^2) \asymp \frac{1}{\sqrt{q_U(x_n)}},
$$
where $x_n$ is given by 
$$\displaystyle \int _{x_n}^\pi\frac{q_U(t)}{U(t)}dt =n.$$
\end{enumerate}
\end{thm}

\subsection{Composition operators on the Dirichlet space}
The Dirichlet space, denoted by $\cD$, is given by 
$$
\cD (:= \cH _0)= \{ f \in H(\DD): \ f' \in  L^2(\DD, dA)\}.
$$
The Nevanlinna counting function $N_{\varphi ,0}$ induced by $\varphi$ and associated with $\cD$ is the counting function $n_\varphi$. That is 
$$
N_{\varphi , 0}(z)= n_\varphi (z) = \text{Card} \{ \varphi ^{-1}(z)\},\quad (z\in \DD).
$$
\indent In particular, if $\varphi$ is univalent then 
$$n_\varphi = \chi _{\Omega}\ \ \mbox{ and}\ \  d\mu_{\varphi, 0}=  \chi _{\varphi (\Omega)}dA, \quad (\Omega = \varphi (\DD)).$$

Let $\Omega, \gamma$ and $\varphi$ as before. The compactness and membership to schatten classes of $C_\varphi$ is studied in \cite {EEN}. Recall that $C_\varphi$ is compact on $\cD$ if and only if 
$$\displaystyle \lim_{t\to 0^+}\frac{\gamma (t)}{t} = \infty.$$
And $C_\varphi \in S_p$ if and only if 
$$
\displaystyle \int _0 \left ( \frac{t}{\gamma (t)}\right )^{p/2}\frac{\gamma '(t)}{\gamma (t)}dt<\infty.
$$
Using Corollary \ref {sComposition} and the discussion above, one can prove easily that if $\gamma (t)/t = O(\log ^\beta (1/t))$ for some $\beta >0$, then 
$$
s_n(C_\varphi) \asymp \sqrt{e^n\gamma ^{-1}(e^{-n})}.
$$

\bibliographystyle{amsplain}
\bibliography{bibl}

\end{document}